%% file: hilb.tex
\documentclass[12pt]{amsart}

\input{preamble}

\title{Hilbert Schemes of $8$ Points}

\author{Dustin A. Cartwright}
\author{Daniel Erman}
\author{Mauricio Velasco}
\author{Bianca Viray}
\address{Department of Mathematics \\ University of California \\ Berkeley, CA
94720 \\ USA}
\email{dustin@math.berkeley.edu, derman@math.berkeley.edu,
velasco@math.berkeley.edu, bviray@math.berkeley.edu}

\thanks{The first author was supported by an NSF EMSW21 fellowship. The second author is supported by an
NDSEG fellowship. The third author is partially supported by NSF grant DMS-0802851. The fourth author was supported by a Mentored Research Award.  }

\begin{document}

\begin{abstract}
The Hilbert scheme $H^d_n$ of $n$ points in $\mathbb A^d$ contains an irreducible component $R^d_n$ which generically represents $n$ distinct points in $\mathbb A^d$. We show that when $n$ is at most $8$, the Hilbert scheme $H^d_n$ is reducible if and only if $n=8$ and $d\geq 4$.  In the simplest case of reducibility, the component $R^4_8 \subset H^4_8$ is defined by a single explicit equation which serves as a criterion for deciding whether a given ideal is a limit of distinct points.

To understand the components of the Hilbert scheme, we study the closed
subschemes of $H_n^d$ which parametrize those ideals which are homogeneous and
have a fixed Hilbert function. These subschemes are a special case of multigraded
Hilbert schemes, and we describe their components when the colength is at most
$8$. In particular, we show that the scheme corresponding to the Hilbert function $(1, 3,2,1)$ is the minimal reducible example.
\end{abstract}
\maketitle

\input{intro}
\input{Background}

\input{stdgraded}
\input{smoothable2}
\input{waldo}

\input{pfofmain}

\input{OpenQuest}

\section*{Acknowledgments}
We thank David Eisenbud, Bjorn Poonen, and Bernd Sturmfels for many insightful discussions, guidance, and key suggestions. We would also like to thank Eric Babson, Jonah Blasiak, Mark Haiman, Anthony Iarrobino, Diane Maclagan, Scott Nollet, and  Mike Stillman for helpful conversations.  Macaulay2~\cite{macaulay2} and Singular~\cite{singular} were used for experimentation.

\bibliographystyle{alpha}
\bibliography{hilb}

\end{document}

%% file: preamble.tex
\usepackage{amsmath,amssymb}
\usepackage{diagrams}
\usepackage {graphicx,enumerate,booktabs}
\usepackage{fullpage,url,amsfonts,amsthm,mathrsfs,diagrams}
\usepackage{color}
\usepackage{amscd}
\usepackage{natbib}
\usepackage{verbatim}
\usepackage{stmaryrd}
\usepackage[bookmarks=true,colorlinks=true]{hyperref}

\makeatletter
\def\subsection{\@startsection{subsection}{2}%
 \z@{.5\linespacing\@plus.7\linespacing}{.3\linespacing}%
 {\normalfont\bfseries\centering}}
\makeatother

\newtheorem{lemma}{Lemma}[section]
\newtheorem{theorem}[lemma]{Theorem}
\newtheorem{propo}[lemma]{Proposition}
\newtheorem{prop}[lemma]{Proposition}
\newtheorem{cor}[lemma]{Corollary}

\newtheorem{claim*}{Claim}
\newtheorem{thm}[lemma]{Theorem}
\newtheorem{defn}[lemma]{Definition}

\theoremstyle{remark}
\newtheorem{remark}[lemma]{Remark}

\newcommand{\isom}{\cong}
\newcommand{\m}{\mathfrak m}
\newcommand{\lideal}{\langle}
\newcommand{\rideal}{\rangle}
\newcommand{\initial}{\operatorname{in}}

\newcommand{\Spec}{\operatorname{Spec}}

\newcommand{\ch}{\operatorname{char}} 
\newcommand{\Proj}{\operatorname{Proj}}

\newcommand{\Pic}{\operatorname{Pic}}

\newcommand{\Hom}{\operatorname{Hom}} 
\newcommand{\sheafHom}{\mathcal{H}om}
\newcommand{\Gr}{\operatorname{Gr}}

\newcommand{\Sym}{\operatorname{Sym}} 
\newcommand{\GL}{{GL}}
\newcommand{\Syz}{\operatorname{Syz}}
\newcommand{\defi}[1]{{\upshape\sffamily #1}}

\numberwithin{equation}{section}
\numberwithin{table}{section}

\newif\ifversion
\versionfalse
\def\cvs $#1 $ $#2 ${\def\cvsversion{\ifversion{\color{red}%
    Last checkin: #2 #1\par}\fi}}

%% file: intro.tex
\ifx\thepage\undefined\def\jobname{hilb}\input{hilb}\fi
\cvs  $Date: 2008/06/19 20:52:38 $ $Author: dustin $
\section{Introduction}\label{sec:intro}\cvsversion

The Hilbert scheme $H^d_n$ of $n$ points in affine $d$-space parametrizes $0$-dimensional, degree~$n$ subschemes of $\mathbb A^d$.  Equivalently, the $k$-valued points of $H^d_n$ parametrize ideals $I\subset S=k[x_1, \dots, x_d]$ such that $S/I$ is an $n$-dimensional vector space over $k$. The \defi{smoothable component} $R^d_n\subset H^d_n$ is the closure of the set of ideals of distinct points.
The motivating problem of this paper is characterizing the ideals which
lie in the smoothable component, i.e.\ the $0$-dimensional subschemes which
are limits of distinct points. We determine the
components of the schemes $H^d_n$ for $n \leq 8$, and find explicit equations defining $R^4_8\subset H^4_8$.

We assume that $k$ is a field of characteristic not~$2$ or~$3$.
\begin{thm}\label{thm:irred}\label{thm:h48red}\label{thm:hd8}
Suppose $n$ is at most $8$ and $d$ is any positive integer. Then the Hilbert scheme $H_n^d$ is reducible if and only if $n = 8$ and $d \geq 4$, in which case it consists of exactly two irreducible components:
\begin{enumerate}
\item the smoothable component, of dimension $8d$
\item a component denoted $G^d_8$, of dimension $8d-7$, which consists
of local algebras isomorphic to homogeneous algebras with
Hilbert function $(1,4,3)$.
\end{enumerate}
\end{thm}

It is known that for $d$ at least $3$ and $n$ sufficiently large the Hilbert scheme of points is always reducible~\cite{iarrobino-red}.
The fact that the Hilbert scheme $H^4_8$ has at least two components appears
in~\cite{iarrobino-small-tangent-space}. 
In contrast, for the Hilbert scheme of points in the plane ($d=2$), the smoothable component is the only component~\cite{fogarty}.

To show that the Hilbert scheme of $n$ points is irreducible, it suffices to
show that each isomorphism type of local algebras of rank at most $n$ is
smoothable, and
for $n$ at most $6$ there are finitely many isomorphism types of local algebras.
In contrast, there are infinitely many non-isomorphic local algebras of degree $7$. Relying on a classification of the finitely many isomorphism types in degree 6, \cite{mazzola} proves the irreducibility of $H^d_n$ for $n=7$.

In our approach, a coarser geometric decomposition replaces most of the need for
classification. We divide the local algebras in $H^d_n$ into sets $H^d_{\vec h}$ by their Hilbert function $\vec h$, and we determine which components of these sets are smoothable.  The main advantage to this approach is that there are fewer Hilbert functions than isomorphism classes, and this enables us to extend the smoothability results of~\cite{mazzola} up to degree $8$.

In order to determine the components of $H^d_{\vec h}$, we first determine the components of the \defi{standard graded Hilbert scheme} $\mathcal H^d_{\vec h}$, which parametrizes homogeneous ideals with Hilbert function $\vec h$.   By considering the map $\pi_{\vec h}\colon H^d_{\vec h} \rightarrow \mathcal H^d_{\vec h}$ which sends a local algebra to its associated graded ring, we relate the components of $\mathcal H^d_{\vec h}$ to those of $H^d_{\vec h}$.  The study of standard graded Hilbert schemes leads to the following analogue of Theorem~\ref{thm:irred}:

\begin{thm}\label{thm:multigrad-red}
Let $\mathcal H_{\vec h}^d$ be the standard graded Hilbert scheme for Hilbert
function $\vec h$, where $\sum h_i\leq 8$.  Then $\mathcal H_{\vec h}^d$ is reducible if and only if $\vec h = (1,3,2,1)$ or $\vec h = (1,4,2,1)$, in which case it
has exactly two irreducible components. In particular, $\mathcal
H^3_{(1,3,2,1)}$ is the minimal example of a reducible standard  graded Hilbert
scheme.
\end{thm}
As in the ungraded case all standard graded Hilbert schemes in the plane are
smooth and irreducible~\cite{evain}.

In the case when $d=4$ and $n=8$, we describe the intersection
of the two components of $H^4_8$ explicitly. Let $S=k[x,y,z,w]$ and $S_1$ be the
vector space of linear forms in $S$. Let $S_2^*$
denote the space of symmetric bilinear forms on $S_1$.
Then, the component $G_8^4$ is isomorphic to $\mathbb A^4 \times \Gr(3,S_2^*)$, where $\Gr(3,S_2^*)$ denotes the Grassmannian of $3$-dimensional subspaces of~$S_2^*$.

\begin{thm}\label{thm:waldo}
The intersection $R^4_8\cap G^4_8$ is a prime divisor on $G^4_8$. We have the following equivalent descriptions of $R_8^4 \cap G_8^4\subset G_8^4$:
\begin{enumerate}
    \item\label{thm:waldo:set} {\bf Set Theoretic} For a point $I \in G^4_8 \cong \mathbb A^4 \times \Gr(3, S_2^*) \cong \mathbb A^4 \times \Gr(7, S_2)$ let $V$ be the corresponding $7$-dimensional subspace of $S_2$.  Then $I\in G^4_8$ belongs to the intersection if and only if the following skew-symmetric bilinear form $\langle , \rangle_I$ is degenerate:
\begin{align*}
\langle , \rangle_I\colon (S_1\otimes S_2/V)^{\otimes 2} &\to \bigwedge^3(S_2/V)\cong k \\
\langle l_1\otimes q_1, l_2\otimes q_2 \rangle_I&=(l_1l_2)\wedge q_1 \wedge q_2
\end{align*}

\item\label{thm:waldo:local} {\bf Local Equations}  Around any $I\in G^4_8$, choose an open neighborhood $U_I\subset G^4_8$ such that the universal Grassmannian bundle over the $U_I$ is generated by three sections.  Since these sections are bilinear forms we may represent them as symmetric $4\times 4$ matrices $A_1, A_2,$ and $A_3$ with entries in $\Gamma(U_I, \mathcal O_{G^4_8})$.  The local equation for $R^4_8\cap U_I$ is then the Pfaffian of the $12\times 12$ matrix:
\begin{equation*}\begin{pmatrix}
0&A_1&-A_2\\
-A_1&0&A_3\\
A_2&-A_3&0
\end{pmatrix}
\end{equation*}
Note that specializing this equation to $I$ gives the Pfaffian of $\langle, \rangle_I$.
\end{enumerate}
\end{thm}
The local equation from the previous theorem gives an effective criterion for deciding whether an algebra of colength $8$ belongs to the smoothable component. Moreover, it can be lifted to equations which cut out $R^4_8\subset H^4_8$. Recall that $H^4_8$ can be covered by open affines corresponding to monomial ideals in $k[x,y,z,w]$ of colength~$8$.
\begin{thm}\label{thm:eqns-radical-component}
On these monomial coordinate charts, $R^4_8\subset H^4_8$ is cut out set-theoretically by
\begin{enumerate}
\item The zero ideal on charts corresponding to monomial ideals with Hilbert functions other than $(1,4,3)$.  
\item The pullback of the equations in Theorem~\ref{thm:waldo} along the projection to homogeneous ideals in charts corresponding to monomial ideals with Hilbert function $(1,4,3)$.
\end{enumerate}
\end{thm}
\begin{remark} It is not known whether $H^4_8$ is reduced. If it is, then the equations in
Theorem~\ref{thm:eqns-radical-component} cut out the smoothable component scheme-theoretically.\end{remark}

The material in this paper is organized as follows: Section~\ref{sec:Background}
contains background and definitions. Section~\ref{sec:stdgraded} describes the
geometry of standard graded Hilbert schemes of degree at most $8$.
Section~\ref{sec:smoothable}
contains proofs of the smoothability of families of algebras and its main steps
are collected in Table~\ref{tbl:thm-1}. Section~\ref{sec:waldo} is devoted to
the study of the components of $H^4_8$ and their intersection.
Section~\ref{sec:pfofmain} ties together these results to give proofs of all theorems mentioned above. Finally, Section~\ref{sec:OpenQuest} proposes some open questions.   

%% file: Background.tex
\ifx\thepage\undefined\def\jobname{hilb}\input{hilb}\fi
\cvs $Date: 2008/06/13 17:19:43 $ $Author: dustin $
\section{Background}\label{sec:Background}\cvsversion

In this section, let $k$ be a field and $S=k[x_1, \ldots, x_d]$.

\subsection{Multigraded Hilbert schemes}

A grading of $S$ by an abelian group $A$ is a semigroup homomorphism $\deg\colon
\mathbb{N}^d\rightarrow A$ which assigns to each monomial in $S$ a degree in
$A$. Let $h\colon A\rightarrow \mathbb{N}$ be an arbitrary function, which we will think of
as a vector $\vec h$, with values $h_a$ indexed by $a$ in $A$. We say that a
homogeneous ideal $I$ in $S$ has Hilbert function $\vec h$ if $S_a/I_a$ has $k$-dimension $h_a$ for all $a\in A$.  
The multigraded Hilbert schemes, introduced by Haiman and Sturmfels, parametrize
homogeneous ideals with a fixed Hilbert function~\cite{haiman-sturmfels}. More precisely these are quasi-projective schemes over $k$ which represent the following functors~\cite[Theorem 1.1]{haiman-sturmfels}:

\begin{defn}\label{defn:multigraded} For a fixed integer $d$, grading $\deg$, and Hilbert function $\vec h$, the \defi{multigraded Hilbert functor} $\mathcal{H}_{\vec h}\colon \underline{k\operatorname{-Alg}}\rightarrow \underline{\operatorname{Set}}$ assigns to each $k$-algebra $T$, the set of homogeneous ideals $J$ in $S\otimes T$ such that the graded component $(S \otimes T/J)_a$ is a locally free $T$-module of rank $h_a$ for all $a$ in $A$.
The \defi{multigraded Hilbert scheme} is the scheme which represents the
multigraded Hilbert functor.

In particular, we will be interested in the following two special kinds of multigraded Hilbert schemes:
\begin{itemize}
\item Let $\deg\colon\mathbb{N}^d\rightarrow 0$ be the constant function to the trivial group and define $h_0=n$. In this case the multigraded Hilbert scheme is  the \defi{Hilbert scheme of $n$ points in $\mathbb A^d$} and will be denoted $H_n^d$.
\item Let $\deg\colon\mathbb{N}^d\rightarrow \mathbb{Z}$ be the summation function, which induces the standard grading $\deg(x_i)=1$. We call the corresponding multigraded Hilbert scheme the \defi{standard graded Hilbert scheme for Hilbert function $\vec h$} and denote it with $\mathcal H^d_{\vec h}$.
\end{itemize}
\end{defn}

If $n=\sum_{j\in \mathbb{N}}h_j$ there is a closed immersion $\mathcal H_{\vec h}^d \rightarrow H_n^d$ by~\cite[Prop. 1.5]{haiman-sturmfels}.

\subsection{Coordinates for the Hilbert scheme of points}  \label{subsec:coordinates}
In this section we briefly discuss some coordinate systems on $H^d_n$. The reader should refer to~\cite[Ch. 18]{cca} for an extended treatment.  For a monomial ideal $M$ of colength $n$ with standard monomials $\lambda$, let $U_{\lambda}\subset H^d_n$ be the set of ideals $I$ such that the monomials in $\lambda$ are a basis for $S/I$. Note that the $U_{\lambda}$ form an open cover of $H^d_n$.  An ideal $I\in U_{\lambda}$ has generators of the form $m-\sum_{m'\in\lambda}c_{m'}^mm'$.
The $c_{m'}^m$ are local coordinates for $U_{\lambda}$ which define a closed immersion into affine space. 

Suppose $V(I)$ consists of $n$ distinct points $q^{(1)},\dots, q^{(n)}$ with coordinates $q_i^{(j)}$ for $1\leq i\leq d$. Fix an order $\lambda=(m_1,\dots, m_n)$ on the set of monomials $\lambda$ and define $\Delta_{\lambda}=\det\left([m_i(q^{(j)})]_{i,j}\right)$. For example, if $\lambda = (1, x_1,\ldots, x_1^{n-1})$, then $\Delta_\lambda$ is the determinant of the Vandermonde matrix on the $q^{(j)}_{1}$. If $I\in U_{\lambda}$, we can express the $c_{m'}^m$ in terms of the $q_i^{(j)}$ using Cramer's rule as: 
\[c_{m'}^m=\frac{\Delta_{\lambda-m'+m}}{\Delta_{\lambda}}\] 
where $\lambda-m'+m$ is the ordered set of monomials obtained from $\lambda$ by
replacing $m'$ with $m$.  Note that the right-hand side of this equality is only
defined for ideals of distinct points.  The quotient or product of two
$\Delta_{\lambda}$s is $S_n$-invariant.  Thus the formula does not depend on the
order of $\lambda$. Gluing over the various $U_\lambda$, these quotients
determine a birational map $(\mathbb{A}^d)^n\mathord\sslash S_n \dasharrow R^d_n$ which is regular when the points $q^{(j)}$ are all distinct. The rational functions $\frac{\Delta_{\eta_1}}{\Delta_{\eta_2}}$ are elements of the quotient field of either $(\mathbb{A}^d)^n\mathord\sslash S_n$ or $R^d_n$.  The expressions $\Delta_\lambda$ and their relationship to the local equations $c_{m'}^m$  were introduced in \cite[Proposition 2.6]{haiman-catalan}.

\subsection{Duality}
First suppose that $k$ has characteristic $0$, and let $S^*$ be the ring $k[y_1, \ldots, y_d]$, with the structure of an $S$-module via formal partial differentiation $x_i \cdot f=\frac{\partial f}{\partial y_i}$. If we look at homogeneous polynomials of a fixed degree $j$ in each of the two rings, we have a pairing $S_j \times S_j^* \rightarrow S_0^* = k$. Any vector subspace of $S_j$ has an orthogonal subspace in $S_j^*$ of complementary dimension. 
In particular, if $I$ is a homogeneous ideal in $S$, we have subspaces $I_j^\perp \subset S_j^*$ and we set $I^\perp = \oplus I_j^\perp$. The subspace $I^\perp$ is closed under differentiation, i.e., $\frac{\partial}{\partial y_i}I_j^\perp \subset I_{j-1}^\perp$ for all $i$ and~$j$.  Conversely, any graded vector subspace $I^\perp \subset S^*$ which is closed under differentiation determines an orthogonal ideal $I\subset S$ with Hilbert function $h_j = \dim_k I_j^\perp$.  Also, note that any linear change of variables in $S$ induces a linear change of variables in $S^*$.

If $k$ has positive characteristic, then the same theory works for sufficiently small degree. Formal partial differentiation gives a perfect pairing $S_j \times S_j^* \rightarrow k$ if and only if $j$ is less than $p$. Thus, we can associate orthogonal subspaces $I_j^\perp$ to a homogeneous ideal $I$ so long as $I_p$ contains all of $S_p$. In this case, we define $I_j^\perp = 0$ for all degrees~$j$ at least~$p$, and $I = \oplus I_j$ as before. Conversely, for a graded vector subspace $I^\perp \subset S_*$ which is closed under differentiation and with $I_j^\perp = 0$ for $j$ at least $p$, the orthogonal space is a homogeneous ideal $I \subset S$ with Hilbert function $h_j = \dim_k I_j^\perp$.

%% file: stdgraded.tex
\ifx\thepage\undefined\def\jobname{hilb}\input{hilb}\fi
\cvs $Date: 2008/06/16 16:02:05 $ $Author: dustin $
\section{Components of the standard graded Hilbert schemes}\label{sec:stdgraded}
\cvsversion

In this section, $k$ will denote an algebraically closed field of characteristic not~$2$ or~$3$.  

We will study the components of the standard graded Hilbert schemes $\mathcal H^d_{\vec h}$ with Hilbert function $\vec h$ where $\sum h_i \leq 8$.  These results will be important for the proofs of smoothability in the following section.  From~\cite[Theorem 1]{evain}, we have that for $d = 2$, the standard graded Hilbert schemes are irreducible. Thus, we will only work with $d$ at least $3$.  For the purposes of classifying irreducible components of $\mathcal H^d_{\vec h}$, it is convenient to work with homogeneous ideals which contain no linear forms, and thus we assume that $h_1=d$. The following lemma allows us to restrict our attention to this case:

\begin{prop} \label{prop:stdgraded-grassmannian-bundle}
The standard graded Hilbert scheme $\mathcal H_{\vec h}^d$ with $d \geq h_1$ is
a $\mathcal H_{\vec h}^{h_1}$-bundle over $\Gr(d-h_1, S_1)$.  In particular, if $\mathcal H^{h_1}_{\vec h}$ is irreducible of dimension~$D$ then $\mathcal H^{d}_{\vec h}$ is irreducible of dimension $D+(d-h_1)d$.
\end{prop}

\begin{proof}
The degree~$1$ summand of the universal ideal sheaf of $\mathcal O_{\mathcal H_{\vec h}^d}[x_1, \ldots, x_d]$ is locally free of rank $d-h_1$ and thus defines a morphism $\phi\colon \mathcal H^d_{\vec h} \to \Gr(d-h_1, S_1)$. Over an open affine $U \isom \mathbb A^{(d-h_1)h_1}$ in $\Gr(d-h_1, S_1)$, we have an isomorphism $\phi^{-1}(U) \isom U \times H_{\vec h}^{h_1}$ by taking a  change of variables in $\mathcal O_U[x_1, \ldots, x_d]$.
\end{proof}

\begin{lemma}
Let be $m$ a positive integer such that $m!$ is not divisible by the characteristic of $k$.  Let $f(y_1, \dots, y_d)$ be a homogeneous polynomial in $S_m^*$ whose partial derivatives form an $r$-dimensional vector subspace of $S_{m-1}^*$. Then $f$ can be written as a polynomial in terms of some $r$-dimensional subspace of $S_1^*$.
\end{lemma}

\begin{proof}
There exists a linear map from $S_1 \to S_{m-1}^*$ which sends $x_i\mapsto \frac{\partial f}{\partial y_i}$.  After a change of variables, we can assume that $x_{r+1}, \dots, x_{d}$ annihilate $f$. Thus, any term of $f$ contains only the variables $y_1, \dots, y_{r}$.
\end{proof}

Throughout this section, $N$ will denote $\dim_k S_2 = {d+1 \choose 2}$, the
dimension of the vector space of quadrics.


\begin{prop}\label{prop:stdgraded-grassmannian}
The standard graded Hilbert scheme for Hilbert function  $(1,d,e)$ is isomorphic to the Grassmannian $\Gr(N - e,S_2)$, and it is thus irreducible of dimension $(N-e)N$.
\end{prop}
\begin{proof}
We build the isomorphism via the functors of points of these schemes.  For a $k$-algebra $T$ let $\phi(T)\colon\mathcal{H}_h(T)\rightarrow \Gr(N-e,S_2)(T)$ be the morphism of sets which maps a homogeneous ideal $I\subset T\otimes_k S$ to $I_2$. Let $\psi(T)\colon\Gr(N-e,S_2)(T)\rightarrow \mathcal{H}_h(T)$ be the map which sends a $k$-submodule $L$ of $T\otimes S_2$ to $L\oplus\bigoplus_{j\geq 3}(T\otimes_k S_j)$, which is an ideal of $T\otimes S$. The natural transformations $\phi$ and $\psi$ are inverses of one another and the isomorphism follows from Yoneda's Lemma.
\end{proof}


\begin{prop} \label{prop:multigrad-1d1s}
Let $\vec h = (1,d,1,\ldots,1)$ and let $m \geq 3$ the largest index such that
$h_m$ is non-zero.
Then the standard graded Hilbert scheme for
$\vec h$ is irreducible of dimension $d-1$.
\end{prop}

\begin{proof}
We claim that the scheme $\mathcal H_{\vec h}^d$ is parametrized by
$\Gr(1, S_1^*)$ by sending a vector space generated by $\ell \in S_1^*$ to the
ideal generated by the quadrics orthogonal to $\ell^2$ and all degree $m+1$
polynomials.  This ideal has the right Hilbert function and the parametrization is clearly surjective.
\end{proof}


\begin{theorem}\label{thm:multigrad-1d21}
If $d$ is at least $3$, the standard graded Hilbert scheme for Hilbert function  $(1,d,2,1)$ is reducible and consists of the following two components:
\begin{enumerate}
\item The homogeneous ideals orthogonal to $\ell^3$, $\ell^2$, and $q$ where $\ell$ is a linear form and $q$ is a quadric linearly independent of $\ell^2$. We denote this component by $\mathcal Q_d$, and $\dim(\mathcal Q_d) = (d^2 + 3d - 6)/2$. 
\item The closure of the homogeneous ideals orthogonal to a cubic $c$ and its partial derivatives, where the degree~$1$ derivatives of $c$ have rank $2$. We denote this component by $\mathcal P_d$, and $\dim(\mathcal P_d) = 2d-1$.
\end{enumerate}
\end{theorem}

\begin{proof}
We compute the dimension of the first component.  It is parametrized by the $1$-dimensional subspace of $S_1^*$ generated by $\ell$ and a $2$-dimensional subspace of $S_2^*$ which contains $\ell^2$. These have dimensions $d-1$ and $N-1-1$ respectively, for a total of $(d^2+3d-6)/2$.

An open subset of the second component, $\mathcal P_d$, is parametrized by a $2$-dimensional subspace $V$ of~$S_1^*$ and a cubic $c \in \Sym_3(V)$ which is not a perfect cube.  The parametrization is by taking the ideal whose components of degrees~$3$ and~$2$ are orthogonal to~$c$ and to its derivatives respectively.  The space of derivatives is $2$-dimensional by our construction of~$c$. The dimension of $\mathcal P_d$ is $3 + 2(d-2) = 2d-1$.

We claim that any homogeneous ideal with Hilbert function $(1,d,2,1)$ lies in one of these two components. Any such ideal is orthogonal to a cubic $c$, and the derivatives of $c$ are at most $2$-dimensional. If the derivatives are $1$-dimensional, then $c$ must be a perfect cube, so the ideal is in $\mathcal Q_d$.  Otherwise, the ideal is in $\mathcal P_d$.

Finally, we will show that $\mathcal P_d$ has a point that does not lie on $\mathcal Q_d$.  Let $I$ be the ideal orthogonal to $x_1x_2^2$, and its partial derivatives, $2x_1x_2$, $x_2^2$. Then $I$ is generated by $x_2^3$ and all degree $2$ monomials other than $x_2^2$ and $x_1x_2$.  We will study the degree $0$ homomorphisms $\phi\colon I\rightarrow S/I$ as these correspond to the tangent space of $\mathcal  H_{\vec h}^d$ at $I$.
For any quadric generator $q$, we can write $\phi(q) = a_q x_2^2 + b_q x_1 x_2$. Note that $x_1 \phi(q) = a_q x_1x_2^2$ and $x_2 \phi(q) = b_q x_1x_2^2$.  For any $i, j > 2$, $\phi$ must satisfy the conditions:
\begin{align*}
x_1 \phi(x_i x_j) &= x_j \phi(x_i x_1) = 0 \\
x_2 \phi(x_i x_j) &= x_j \phi(x_i x_2) = 0 \\
x_1 \phi(x_1 x_i) &= x_i \phi(x_1^2) = 0 \\
x_2 \phi(x_1 x_i) &= x_1 \phi(x_2 x_i)
\end{align*}
In matrix form, we see that $\phi$ must be in the following form:
\begin{equation*}
\bordermatrix{
& x_1^2  & x_1 x_i & x_2 x_i & x_i x_j & x_2^3 \cr
x_2^2 & * & c_i & * & 0 & 0 \cr
x_1x_2 & * & 0 & c_i & 0 & 0 \cr
x_1x_2^2 & 0 & 0 & 0 & 0 & *
}
\end{equation*}
where $i$ and $j$ range over all integers greater than $2$.  Thus there are at most $2(d-2) + 3= 2d-1$ tangent directions, but since $\mathcal P_d$ has dimension $2d-1$, there are exactly $2d-1$ tangent directions. On the other hand, $\mathcal Q_d$ has dimension $(d^2+3d-6)/2$ which is greater than $2d-1$ for $d$ at most $3$, so $I$ cannot belong to $\mathcal Q_d$ and thus $\mathcal P_d$ is a component.
\end{proof}


\begin{prop} \label{prop:stdgraded-1d22}
The standard graded Hilbert scheme for Hilbert function $(1,d,2,2)$ is irreducible of dimension  $2d - 2$. 
\end{prop}

\begin{proof}
The Hilbert scheme is parametrized by a $2$-dimensional subspace $L$ of $S_1^*$ and a subspace of $S_2$ of dimension $N-2$, and containing the $(N-3)$-dimensional subspace orthogonal to the square of $L$.
The parametrization is by sending the subspace of $S_2$ to the ideal generated by that subspace, together with all degree~$4$ polynomials.  
The dimension of this parametrization is $2(d-2) + 2 = 2d-2$.
\end{proof}


\begin{prop} \label{prop:stdgraded-1331}
The standard graded Hilbert scheme for Hilbert function $(1,3,3,1)$ is irreducible of dimension $9$.
\end{prop}

\begin{proof}
This Hilbert scheme embeds as a closed subscheme of the smooth $18$-dimensional variety $\Gr(3, S_2) \times \Gr(9, S_3)$ by mapping an ideal to its degree $2$ and $3$ graded components. Furthermore, $\mathcal H_{\vec h}^d$ is defined by $9 = 3 \cdot 3$ equations, corresponding to the restrictions that the products of each of the $3$~variables and each of the $3$~quadrics in $I_2$ are in~$I_3$. In particular, the dimension of each irreducible component is at least $9$.

Now we will look at the projection of $\mathcal H_{\vec h}^d$ onto the Grassmannian $\Gr(9, S_3)$, which is isomorphic to $\Gr(1, S_3^*)$. The orthogonal cubic in $S_3^*$ can be classified according to the vector space dimension of its derivatives. For a generic cubic, its three derivatives will be linearly independent and therefore the cubic will completely determine the orthogonal space. Thus, the projection from $\mathcal H_{\vec h}^d$ is a bijection over this open set, so the preimage is $9$-dimensional. In the case where the derivatives of the cubic are $2$-dimensional, we have that, after a change of coordinates, the cubic is written in terms of two variables. Thus, the parameter space of the cubic consists of a $2$-dimensional choice of a subspace of~$S_1$ and then a $3$-dimensional choice of a cubic written in terms of this subspace. The fiber over any fixed cubic is isomorphic to the Grassmannian of $3$-dimensional subspaces of the $4$-dimensional subspace of $S_2$ orthogonal to the derivatives of the cubic. The dimension of the locus in $\mathcal H_{\vec h}^d$ is therefore $2+3+3=8$. By a similar logic, the locus where the cubic has a $1$-dimensional space of derivatives is $2+2\cdot 3=8$. Therefore, $\mathcal H_{\vec h}^d$ is the disjoint union of three irreducible sets, of dimensions $9$, $8$, and $8$.  We conclude that $\mathcal H^d_{\vec h}$ is an irreducible complete intersection of dimension $9$.
\end{proof}


\begin{prop} \label{prop:stdgraded-1de11}
The standard graded Hilbert scheme for Hilbert function $(1,d,e, 1,1)$
is irreducible of dimension $d-1 + (N-e)(e-1)$.
\end{prop}

\begin{proof}
This Hilbert scheme is parametrized by a $1$-dimensional subspace $L$ of $S_1^*$, together with an $e$-dimensional subspace $V$ of $S_2^*$ which contains $\Sym_2(L)$. The parametrization is by mapping $(L, V)$ to the ideal whose summands of degrees $2$, $3$, and $4$ are orthogonal to $V$, $L^3$, and $L^4$, respectively.  Note that this has the desired dimension $(d-1) + ((N-1)-(e-1))(e-1)$. 
%
\end{proof}

\begin{thm} \label{thm:summary-stdgraded}
With the exception of Hilbert function $(1,3,2,1)$ and $(1,4,2,1)$, the standard graded Hilbert schemes with $\sum h_i\leq 8$ are irreducible.
\end{thm}

\begin{proof}
The cases when $d=2$ follow from~\cite[Theorem~1]{evain}. The cases when $d$ is at least $3$ are summarized in Table~\ref{tbl:thm-1}
\end{proof}

%% file: smoothable2.tex
\ifx\thepage\undefined\def\jobname{hilb}\input{hilb}\fi
\cvs  $Date: 2008/06/22 10:30:12 $ $Author: dustin $
\section{Smoothable 0-schemes of degree at most 8}\label{sec:smoothable}\cvsversion

In this section $k$ will denote an algebraically closed field of characteristic not~$2$ or~$3$.

Recall that a point $I$ in $H^d_n$ is \defi{smoothable} if $I$ belongs to the smoothable component $R^d_n$.  In this section, we first reduce the question of smoothability to ideals $I$ in $H^d_n$ where $S/I$ is a local $k$-algebra and $I$ has embedding dimension $d$.   Then we define the schemes $H^d_{\vec h}$ which parametrize local algebras, and we use these to show that each $0$-dimensional algebra of degree at most 8 is either smoothable or is isomorphic to a homogeneous local algebra with Hilbert function $(1,4,3)$.

We use two different methods to show that a subscheme $H^d_{\vec h}$ belongs to the smoothable component.
\begin{enumerate}
	\item For each irreducible component of $H^d_{\vec h}$, consider a generic ideal $I$ from that component.  Apply suitable isomorphisms to put $I$ into a nice form.  Then show $I$ is smoothable.  Since the set of ideals isomorphic to $I$ are dense in the component and smoothable, the entire component of  $H^d_{\vec h}$ containing $I$ must belong to $R^d_n$.
	\item Within each irreducible component of $H^d_{\vec h}$, find an ideal $I$ such that $I$ is a smooth point in $H^d_n$ and $I$ belongs to $R^d_n$.  Then the whole component of $H^d_{\vec h}$ containing $I$ must belong to $R^d_n$.
\end{enumerate}

In each method we need to show that a particular ideal $I$ is smoothable.  We do this by showing $I=\initial J$ with respect to some non-negative weight vector for a smoothable ideal $J$.  The corresponding Gr\"obner deformation induces a morphism $\mathbb A^1\to R^d_n$ which maps $0$ to $I$. 

For $d=2$, we have the following:
\begin{thm}[Fogarty] \label{thm:plane}
The Hilbert scheme $H_n^2$ is smooth and irreducible.
\end{thm}
\begin{proof}
Theorem~2.4 of~\cite{fogarty}. For a different proof, see Theorem~18.7
of~\cite{cca} which is based on Proposition~2.4 of~\cite{haiman-catalan}.
\end{proof}
Thus, we will limit our analysis to algebras with embedding dimension at
least~$3$.

\subsection{Reduction to local Artin \texorpdfstring{$k$}{k}-algebras}

By the following lemma, it suffices to consider only ideals
supported at a single point:
\begin{lemma} \label{lem:primary}
Let $I$ be an ideal in $A = k[x_1, \ldots, x_d]$ of colength $n$ with a decomposition $I = I_1 \cap \cdots \cap I_m$ where the $I_i$ are pairwise coprime. Then $I \in H_n^d$ is smoothable if each $I_i \in H_{n_i}^d$ is smoothable, where $n_i$ is the colength of $I_i$.
\end{lemma}

\begin{proof}
It suffices to prove this in the case $m=2$. We will construct
a rational map from $P := H^d_{n_1}\times H^d_{n_2}$ to $H^d_{n}$
which sends $R^d_{n_1}\times R^d_{n_2}$ to $R^d_{n}$. Consider the universal
ideal sheaves on $H^d_{n_1}$ and~$H^d_{n_2}$, and let $\mathcal I_1$ and
$\mathcal I_2$ be their pullbacks to ideal sheaves of $\mathcal O_P[x_1, \ldots,
x_d]$. Define $U \subset P$ to be the complement of the
support of the cokernel of $\mathcal I_1 + \mathcal I_2$. On $U$, the ideal
sheaf $\mathcal I_1 \cap \mathcal I_2$ has cokernel which is locally free of
rank $n_1 + n_2$ by the Chinese remainder theorem. Thus, it defines a map
$\phi\colon U \rightarrow H^d_n$. Since $\phi$ takes radical ideals to radical
ideals, and $(I_1, I_2)$ to $I_1 \cap I_2$, the result follows.
\end{proof}

Furthermore, by the following lemma, it suffices to consider isomorphism
classes of the quotient algebra.
\begin{lemma} \label{lem:change-of-var}
Let $I$ be an ideal in $S = k[x_1,\ldots,x_d]$ and $J$ an ideal in $T = k[y_1,\ldots, y_{d'}]$.  
Suppose that $S/I$ are $T/J$ isomorphic as $k$-algebras. Then
there exists an open neighborhood $U$ of $J$ in $H_n^{d'}$ and a morphism
from $U$ to $H_n^d$ sending $J$ to $I$ and such that the pullbacks of the universal
sheaves of algebras on $H_n^d$ and $H_n^{d'}$ are isomorphic as $\mathcal
O_U$-algebras.
\end{lemma}

\begin{proof}
Let $f\colon S \rightarrow T$ be a ring homomorphism defined by sending each
$x_i$ to some lift of its image under some fixed isomorphism $T/J \isom S/I$. We will use this to construct
a map from an open subset of $H = H_n^{d'}$ to $H_n^{d}$ which is conceptually
the map taking an ideal to its preimage under $f$.  Consider the sheaf of
algebras $\mathcal O_H[y_1, \ldots, y_{d'}]$ with the universal sheaf of
ideals $\mathcal I$, and let $\mathcal A$ denote the cokernel, which is locally
free of rank $n$. Let $M$ be any set of $n$ monomials in $S$ which form a $k$-basis
for $S/I$. Define $U$
to be the open set where the images of the elements $f(M)$ generate $\mathcal A$
(in particular this implies that $f(M)$ are distinct).
The $k$-algebra isomorphism between $S/I$ and $T/J$ guarantees that $U$ contains
at least $J$. 

Now consider the sheaf of algebras $\mathcal O_H[x_1, \ldots , x_d]$. The
ring homomorphism $f$ induces a sheaf homomorphism $\mathcal O_H[x_1, \ldots, x_d]$ to
$\mathcal O_H[y_1, \ldots, y_{d'}]$, which will also be denoted $f$. Let
$\mathcal B$ be the cokernel of the sheaf of ideals $f^{-1}(\mathcal I)$:
\begin{equation} \label{eqn:cov-diagram}
\begin{CD}
0 @>>>f^{-1}(\mathcal I) @>>> \mathcal O_H[x_1, \ldots, x_d] @>>> \mathcal B
@>>> 0 \\
@. @. @VVfV @VV\overline fV \\
0 @>>> \mathcal I @>>> \mathcal O_H[y_1, \ldots y_{d'}] @>>> \mathcal A @>>>
0
\end{CD}
\end{equation}
We claim that $\mathcal B|_U$ is free and in fact $\overline f|_U$ is an
isomorphism. Since $\mathcal A|_U$
is a free $\mathcal O_{U}$-module with generating set $f(M)$, we
can define a left inverse of~$\overline f$ by sending each element of the basis set $f(M)$ to the
corresponding element of $M$. Thus, $\overline f|_U$ is a surjection. By
the construction of $\mathcal B$, $\overline f|_U$ is an injection, so it is an
isomorphism and in fact an isomorphism of $\mathcal O_U$-algebras. From the
universal property of the Hilbert scheme, $f^{-1}(\mathcal I)|_U$ induces a
morphism $\phi\colon U \rightarrow H_n^d$.

We claim that $\phi(J) = I$.  Look at the fiber of the
diagram~\ref{eqn:cov-diagram} over the point corresponding to $J$, which remains
exact since all of the sheaves are locally free, and we see that $\phi(J) =
f^{-1}(J) = I$.  Finally, $\overline f|_U$ is exactly the isomorphism of
in the statement.
\end{proof}

\begin{cor} \label{cor:change-of-vars-smoothable}
If $I$ and $J$ are as in Lemma~\ref{lem:change-of-var} and $J$ is smoothable, then so is $I$.
\end{cor}

\begin{proof}
Let $\phi$ be the morphism from the lemma. The isomorphism of sheaves of algebras means that $\phi$ sends ideals of distinct points to ideals of distinct points.  By continuity, $I$ is in the smoothable component.
\end{proof}

Because of this, we will speak interchangeably of a point in the Hilbert scheme, an ideal in the polynomial ring, and its quotient algebra, and we will call a $k$-algebra \defi{smoothable} if any ideal defining it is smoothable.  From now on, we consider only ideals $I$ in $H^d_n$ which define local algebras and with embedding dimension $d$.

\subsection{The schemes \texorpdfstring{$H^d_{\vec h}$}{Hd}}\label{subsec:Hdh}
If $(A, \mathfrak m)$ is a local algebra, its Hilbert function is defined by
$h_i = \dim_k \mathfrak m^i/\mathfrak m^{i+1}$, which is equivalently the
Hilbert function of the associated graded ring of $A$.
When $A$ is both local and graded, the two notions of Hilbert function coincide.
We now define the schemes $H^d_{\vec h}$ and explore their irreducible
components for each Hilbert function $\vec h$ with $\sum h_i\leq 8$. 

For each $\vec h$ such that $\sum h_i=n$, the subscheme $H^d_{\vec h}\subset
H^d_n$ consists set-theoretically of the ideals $I$ defining a local
algebra $S/I$ with maximal ideal $(x_1,\dots, x_d)$ whose Hilbert function
equals $\vec h$.  More precisely, let $\mathcal A = \mathcal
O_{H_n^d}[x_1,\ldots,x_d]/\mathcal I$ be the universal sheaf of algebras on
$H_n^d$ and let $\mathcal M$ be the ideal $(x_1, \ldots, x_d)\mathcal A$.
The fiber at an ideal $I$ of the quotient sheaf $\mathcal A / \mathcal M^i$ is
isomorphic to $S/(I + (x_1, \ldots, x_d)^i)$. For any fixed $\vec h$, there is a
locally closed subset of $H^d_n$ consisting of those point such that the fiber
of $\mathcal A / \mathcal M^i$ has dimension $h_0 + \ldots + h_{i-1}$ for all $i
\geq 0$.
Let $H_{\vec h}^d$ be the reduced subscheme 
on this subset, and then the restriction of each $\mathcal A /
\mathcal M^{i}$ to $H_{\vec h}^d$ is locally free. Define $\mathcal B$ to be the
sheaf of graded algebras on $H_{\vec h}^d$ whose $i$th component is:
\begin{equation*}
\ker\left((\mathcal A/\mathcal M^{i+1})|_{H_{\vec h}^d}
\rightarrow (\mathcal A / \mathcal M^{i})|_{H_{\vec h}^d} \right)
\end{equation*}
which is locally free of rank $h_i$ because it is the kernel of a surjection of
locally free sheaves. Note that the fiber of $\mathcal B$ at $I$ is the
associated graded ring of $S/I$.
There is a canonical surjection of graded algebras $O_{H_{\vec h}^d}[x_1,
\ldots, x_d] \rightarrow
\mathcal B$ which defines a morphism
$\pi_{\vec h}\colon H_{\vec h}^d \rightarrow \mathcal H_{\vec h}^d$ to the
standard graded Hilbert scheme.   The ideal $I$ gets mapped to its initial ideal
with respect to the weight vector $(-1,\ldots, -1)$.

With the exception $\vec h=(1,3,2,1,1)$, we will show that the irreducible
components of~$H_{\vec h}^d$ and~$\mathcal H_{\vec h}^d$ are in bijection via
the map $\pi_{\vec h}$.

\begin{prop}\label{prop:1de-irred}
	Each subscheme $H^d_{(1,d,e)}$ is irreducible.
\end{prop}
\begin{proof}
	Since $H^d_{(1,d,e)}\cong \mathcal H^d_{(1,d,e)}$, this follows from Proposition~\ref{prop:stdgraded-grassmannian}.
\end{proof}

\begin{prop} \label{prop:fibration}
Fix $\vec h = (1,d,e, f)$. Let $m = (d+1)d/2 - e = \dim_k S_2/I_2$.  Then every fiber of $\pi_{\vec h}$ is irreducible of dimension $mf$. In particular, the irreducible components of $H_{\vec h}^d$ are exactly the preimages of the irreducible components of $\mathcal H_{\vec h}^d$.
\end{prop}

\begin{proof}
Fix a point in $\mathcal H_{\vec h}^d$, which corresponds to a homogeneous ideal
$I$. Let $q_1, \ldots, q_m$ be quadratic generators of $I$, and let $c_1,
\ldots, c_f$ be cubics which form a vector space basis for $S_3/I_3$.  Define a map $\phi\colon\mathbb A^{mf} \rightarrow H_n^d$ via the ideal
\begin{equation*}
\left\lideal q_i - \textstyle\sum_{j=1}^f t_{ij} c_j \mid 1 \leq i \leq m \right\rideal +
I_{\geq 3}
\end{equation*}
where the $t_{ij}$ are the coordinate functions of $\mathbb A^{mf}$.  Because a product of any variable $x_\ell$ with any of these generators is in $I$, this ideal has the right Hilbert function and maps to the fiber of $\pi_{\vec h}$ over $I$.  Furthermore, $\phi$ is bijective on field-valued points, so the fiber is irreducible of dimension $mf$.  

For the last statement, we have that for any irreducible component of $\mathcal H_{\vec h}^d$, the restriction of $\pi_{\vec h}$ has irreducible equidimensional fibers over an irreducible base, so the preimage is irreducible. These closed sets cover $H_{\vec h}^d$ and because each lies over a distinct component of $\mathcal H_{\vec h}^d$, they are distinct irreducible components.
\end{proof}

Combining Theorem~\ref{thm:multigrad-1d21} with the above proposition, we see $H^d_{(1,d,2,1)}$ has exactly two components: $P_d:=\pi^{-1}(\mathcal P_d)$ and $Q_d:=\pi^{-1}(\mathcal Q_d)$.  In addition, by Propositions~\ref{prop:stdgraded-1d22}, \ref{prop:stdgraded-1331}, and~\ref{prop:fibration}, $H^d_{(1,d,2,2)}$ and $H^3_{(1,3,3,1)}$ are irreducible.

\begin{prop}\label{prop:1d11-irred}
Let $\vec h = (1,d,1,\ldots,1)$ and let $m \geq 3$ be the largest index such
that $h_m$ is non-zero. Then $H^d_{\vec h}$ is irreducible of dimension
$(d+2m-2)(d-1)/2$. At a generic point, after a change of coordinates, we can
take the ideal to be:
\begin{equation*}
\lideal x_d^{m+1}, x_i^2 - x_d^{m}, x_j x_k \mid 1 \leq i < d, 1 \leq j < k \leq
d\rideal
\end{equation*}
\end{prop}	

\begin{proof}
Fix an ideal $I \in \mathcal H^d_{\vec h}$, and after a change of
coordinates, we can assume
\begin{equation*}
I = \lideal x_d^{m+1}, x_ix_j \mid 1 \leq i \leq d-1, 1 \leq j \leq d \rideal
\end{equation*}
Let $J$ be an ideal in the fiber above $I$.
By assumption $J$ contains an elements of the form
\begin{equation*}
x_ix_d - b_{i3} x_d^3 - \cdots - b_{im} x_d^{m}
\end{equation*}
for $1 \leq i < d$.
Let $J'$ be the image of $J$ after the change of coordinates
\begin{equation} \label{eqn:1d11-cov}
x_i \mapsto x_i + b_{i3} x_d^2 + \cdots b_{im} x_d^{m-1}
\end{equation}
and note that $J'$ contains $x_i x_d$ for $1 \leq i < d$ and also lies over $I$.
Thus, for each $1 \leq i \leq j < d$, $J'$ contains an element of the form $f =
x_i x_j - a_{ij} x_d^k - \cdots$ for some $k$. However, $J'$ must also contains $x_j (x_i
x_d) - x_d f = a_{ij} x_d^{k+1} + \cdots$, so in order to have $I$ as the
initial ideal, $k$ must equal $m$. Therefore, $J'$ is of the form
\begin{equation*}
J' = \lideal x_d^{m+1}, x_ix_j - a_{ij} x_d^{m}, x_k x_d \mid 1 \leq i \leq j
\leq d, 1 \leq k < d \rideal
\end{equation*}

Conversely, for any choice of $a_{ij}$ and $b_{ij}$, applying the change of
variables in Equation~\ref{eqn:1d11-cov} to the ideal $J'$ gives a unique ideal
$J$ with $I$ as an initial
ideal.  Thus, the fiber is irreducible of dimension $(m-2)(d-1) + (d-1)d/2 =
(d-1)(d + 2m-4)/2$, which, together with Proposition~\ref{prop:multigrad-1d1s}
proves the first statement.

For the second statement, note that the coefficients $a_{ij}$ define a symmetric
bilinear form. By taking the form to be generic and choosing a change of
variables, we get the desired presentation of the quotient algebra.
\end{proof}

The above propositions cover all Hilbert functions of length at most $8$ except for $\vec h=(1,3,2,1,1)$.  In this case the fibers of $\pi_{(1,3,2,1,1)}$ are not equidimensional.  The dimension of the fiber depends on whether or not the homogeneous ideal requires a cubic generator.  

\begin{lemma}
No ideal in $\mathcal H^3_{(1,3,2,1,1)}$ requires a quartic generator.	
\end{lemma}

\begin{proof}
If $I$ were such an ideal, then leaving out the quartic generator would yield an
ideal with Hilbert function $(1,3,2,1,2)$. No such ideal exists, because no such
monomial ideal exists.
\end{proof}

\begin{lemma} \label{lem:13211-stratification}
There is a $4$-dimensional irreducible closed subset $\mathcal Z$ of $\mathcal H = \mathcal H_{(1,3,2,1,1)}^3$ where the corresponding homogeneous ideal requires a single cubic generator.  On $\mathcal U = \mathcal H \backslash \mathcal Z$, the ideal does not require any cubic generators.
\end{lemma}

\begin{proof}
Let $\mathcal S_j$ denote the $j$th graded component of $\mathcal O_{\mathcal H}[x, y,z]$ and $\mathcal I_j \subset \mathcal S_j$ the $j$th graded component of the universal family of ideals on $\mathcal H$.
Consider the cokernel $\mathcal Q$ of the multiplication map on the coherent sheaves $\mathcal I_2 \otimes_{\mathcal O_{\mathcal H}} \mathcal S_1 \rightarrow \mathcal I_3$ on $\mathcal H$.  The dimension of $\mathcal Q$ is upper semicontinuous.  Furthermore, since it is not possible to have an algebra with Hilbert function $(1,3,2,3)$, the dimension is at most $1$. The set $\mathcal Z$ is exactly the support of $\mathcal Q$.

We claim that $\mathcal Z$  is parametrized by the data of a complete flag $V_1
\subset V_2 \subset S_1^*$, and a $2$-dimensional subspace $Q$ of $V_2^2$ which
contains $V_1^2$. The dimension of this parametrization is $2+1+1 = 4$. An ideal
is formed by taking the ideal which is orthogonal to $Q$ in degree $2$ and to
the powers $V_1^3$ and $V_1^4$ in degrees $3$ and $4$ respectively. After a change of variables, we can assume that the flag is orthogonal to $\langle x \rangle \subset \langle x, y\rangle \subset S_2$.  Then the degree $2$ generators of $I$ are $x^2, xy, xz$ and another quadric. It is easy to see that these only generate a codimension $2$ subspace of $S_3$.  Conversely, for any ideal with this property, the orthogonal cubic has a $1$-dimensional space of derivatives. Furthermore, there exists a homogeneous ideal with Hilbert function $(1,3,2,2,1)$ contained in the original ideal. The cubics orthogonal to these have a $2$-dimensional space of derivatives, so we can write them in terms of a $2$-dimensional space of the dual variables. These two vector spaces determine the flag, and the parametrization is bijective on closed points. In particular, $\mathcal Z$ is irreducible of dimension $4$.
\end{proof}

\begin{lemma} \label{lem:13211z-irred}
The preimage $Z := \pi^{-1}(\mathcal Z)$ is irreducible of dimension $11$.
\end{lemma}

\begin{proof}
By Lemma~\ref{lem:13211-stratification}, it suffices to prove that the fibers of $\pi$ over $\mathcal Z$ are irreducible and $7$-dimensional.  Let $I$ be a point in $\mathcal Z$. As in the proof of Lemma~\ref{lem:13211-stratification}, we can assume that the ideal corresponding to a point of $\mathcal Z$ is generated by $x^2, xy, xz, q, c$, where $q$ and $c$ are a homogeneous quadric and cubic respectively. A point $J$ in the fiber must be generated by $\mathfrak m^5$ and:
\begin{align*}
g_1 &:= x^2 + a_1 z^3 + b_1z^4 \\
g_2 &:= xy + a_2z^3 + b_2z^4 \\
g_3 &:= xz + a_3z^3 + b_3z^4 \\
g_4 &:= q + a_4z^3 + b_4z^4 \\
g_5 &:= c + b_5z^4
\end{align*}

The $a_i, b_i$ are not necessarily free.  We must impose additional conditions to ensure the initial ideal for the weight vector $(-1,-1,-1)$ is no larger than $I$.  In particular we must have 
\begin{align*}
	z g_1 - x g_3 & =  a_1z^4 + b_1z^5 - a_3xz^3 - b_3xz^4 \in J \\
	z g_2 - y g_3 & =  a_2z^4 + b_2z^5 - a_3yz^3 - b_3yz^4 \in J 
\end{align*}	

This implies $a_1=a_2=0$, because the final three terms of each expression are already in $J$.  By Buchberger's criterion, it is also sufficient for these conditions to be satisfied. Therefore the fibers are $7$-dimensional.
\end{proof}

\begin{lemma} \label{lem:13211u-irred}
The preimage $U := \pi^{-1}(\mathcal U)$ is irreducible of dimension $12$.
\end{lemma}

\begin{proof}
By Proposition~\ref{prop:stdgraded-1de11}, it suffices to show that the fibers of $\pi$ over $\mathcal U$ are irreducible of dimension $6$.  Let $I$ be an ideal in $\mathcal U$. Let $V$ be the $1$-dimensional subspace of $S_1^*$ such that $I_3$ is orthogonal to $\Sym_3(V)$ and let $q_1, \ldots, q_4$ be the degree~$2$ generators of $I$. Choose a basis $x,y,z$ of $S_1$ such that $x, y$ is a basis for $V^\perp$.  Then any $J$ in $\phi^{-1}(I)$ is of the form $$\lideal q_i + a_iz^3 +b_iz^4 \mid 1 \leq i \leq 4\rideal + \m^5.$$

	
We claim that forcing $\initial_{(-1,-1,-1)}(J) = I$ imposes two linear conditions on the $a_i$s.  Using the table of isomorphism classes of $(1,3,2)$ algebras in~\cite{poonen-alg-isom}, one can check that for any $4$-dimensional subspace $\lideal q_1,q_2,q_3,q_4\rideal$ of $\Sym_2(V)$, the intersection of $\lideal zq_i \mid 1\leq i\leq 4 \rideal_3$  and $\lideal xq_j, yq_j\mid 1\leq j\leq 4 \rideal_{3}$ is $2$-dimensional.  After choosing a different basis for $Q$, we may assume $zq_1, zq_2 \in \lideal xq_j, yq_j \mid 1\leq j\leq 4\rideal$.  By using a similar argument to the one in Lemma~\ref{lem:13211z-irred}, we see $a_1=a_2=0$.  Since the only other linear syzygies among the $q_i's$ have no $z$ coefficients and $xz^3,yz^3$ and $\m^5$ are in the ideal, these are the only conditions imposed.  Therefore, the fiber is $6$-dimensional.
\end{proof}

Therefore, it suffices to show the following irreducible sets are contained in the smoothable component.
$$ H^d_{(1,d,1,\ldots,1)}, H^d_{(1,d,2)}, P_d,Q_d, H^d_{(1,d,2,2)}, H^3_{(1,3,4)}, H^3_{(1,3,3)},H^3_{(1,3,3,1)}, U,Z.$$

\subsection{Smoothable generic algebras}

In this section we prove that the irreducible sets:
$$H^d_{(1,d,1,\ldots,1)}, H^d_{(1,d,2)}, P_d,Q_d, H^d_{(1,d,2,2)}, H^3_{(1,3,4)}.$$
are in the smoothable component by showing that a generic algebra in each
is smoothable.

\begin{prop}\label{prop:1d1-smoothable}
	All algebras in $H^d_{(1,d,1,\ldots,1)}$ are smoothable.
\end{prop}
\begin{proof}
We prove this by induction on $d$.  Note the $d=1$ case is trivial.
Let $m$ be the greatest integer such that $h_m$ is
nonzero. Then, by Proposition~\ref{prop:1d11-irred} we can take a generic ideal
to be
	$$I=\langle x_1^2-x_d^{m}, \dots, x_{d-1}^2-x_d^{m}, x_d^{m+1}\rangle + \langle x_ix_j \mid i\ne j \rangle.$$
	We define $J$ to be:
	\begin{equation*}
	J=\langle x_1^2+x_1-x_d^{m}, \dots, x_{d-1}^2-x_d^{m}, x_d^{m+1}\rangle + \langle x_ix_j \mid i\ne j \rangle
	\end{equation*}
	Note that $J$ admits a decomposition as $J=J_1\cap J_2$ where $J_1=\langle x_1+1,x_2,x_3,\dots x_d \rangle$ and
	\begin{equation*}
	J_2= \langle x_1-x_d^{m}, x_2^2-x_d^{m}, \dots, x_{d-1}^2-x_d^{m}, x_d^{m+1}\rangle + \langle x_ix_j \mid i\ne j \rangle
	\end{equation*}
	As the Hilbert function of $J_2$ equals $(1,d-1,1,\ldots,1)$, the inductive hypothesis implies that $J_2$ is smoothable.  Thus $J$ itself is also smoothable.  Next note that $I\subset \initial_{(m,\ldots,m,2)}(J)$.  Since both $I$ and $J$ have the same colength, we obtain the equality $I= \initial_{(m,\ldots,m,2)}(J)$.  The corresponding Gr\"obner degeneration induces a map $\mathbb A^1\to R^d_n$ which sends $0$ to $I$.  Thus $I$ is smoothable.
\end{proof}

\begin{prop}\label{prop:1d2-smoothable}
	All algebras in $H^d_{(1,d,2)}$ are smoothable.
\end{prop}
\begin{proof}
	The proof is by induction on $d$.  The case $d=2$ follows from Theorem~\ref{thm:plane}.

        Assume $d$ is at least $3$. Note that $I_2^\perp$ defines a pencils of
        quadrics in $d$-variables.  It then follows from \cite[Lemma~22.42]{harris} that, up to isomorphism, a generic ideal in $H^d_{(1,d,2)}$ is of the form
		\begin{equation*}
        I = \langle x_i x_j \mid i \neq j\rangle + \langle x_i^2 - a_i x_{d-1}^2
        - b_i x_d^2 \mid 1 \leq i \leq d-2 \rangle
        \end{equation*}
with $a_i$ and $b_i$ elements of $k$.

Define 
\begin{align*}
        J_1& = \langle x_i x_j \mid i \neq j\rangle +
		\langle x_1 + a_1 x_{n-1}^2 + b_1 x_d^2, x_i^2 - a_i x_{d-1}^2 -
                b_i x_d^2 \mid 2 \leq i \leq d-2 \rangle \\
		J_2 &= \langle x_1-1, x_2, \dots, x_d\rangle
\end{align*}
Since $J_1$ has Hilbert function $(1,d-1,2)$, it is thus smoothable by the induction hypothesis.  One can check that $I=\initial_{(1,\ldots,1)}\left( J_1\cap J_2\right)$.   Therefore $I$ is smoothable.
\end{proof}	
	
\begin{prop}\label{prop:S1d21-smoothable}
	All algebras in $P_d$ are smoothable.
\end{prop}
\begin{proof}
	Let $I$ be a generic ideal in $P_d$.  
	After a change of variables we may assume $$I = \lideal x_ix_j, x_\ell^2 + x_1^3,x_1^3-x_2^3 \mid 1\leq i<j\leq d, 2<\ell\leq d\rideal.$$ 

One can check 
$$I=\initial_{(2,2,3,\ldots,3)}\left(\lideal x_1+1, x_j \mid j>1\rideal \cap
\lideal x_ix_j, x_\ell^2 +x_1^2,x_1^2-x_2^3 \mid 1\leq i<j\leq d, 2<\ell\leq
d\rideal\right).$$  The second ideal in the intersection has Hilbert function $(1,d,1,1)$, hence is smoothable by Proposition~\ref{prop:1d1-smoothable}.  It follows that $I$ is  smoothable.
\end{proof}

\begin{prop}\label{prop:B1d21-smoothable}
	All algebras in $Q_d$ are smoothable.
\end{prop}

\begin{proof}
	Let $I$ be a generic ideal in $Q_d$.  After a change of variables, we may assume $$I=\lideal x_1x_\ell, x_ix_j+b_{(i,j)}x_1^3, x_k^2-x_{k+1}^2+b_kx_1^3 \mid \ell\neq1, 1<i<j\leq d, 1<k<d \rideal.$$.  Define 
	\begin{eqnarray*}
	J_1&=&\lideal x_1x_\ell, x_ix_j+b_{(i,j)}x_1^2, x_k^2-x_{k+1}^2+b_kx_1^2 \mid \ell\neq1, 1<i<j\leq d, 1<k<d \rideal \\
	J_2&=&\lideal x_1+1,x_2,\ldots, x_d\rideal
	\end{eqnarray*}
	Then $J_1$ has Hilbert function $(1,d,2)$ so is smoothable by Proposition~\ref{prop:1d2-smoothable}.  One can check that $I=\initial_{(2,3,\ldots,3)}\left(J_1\cap J_2\right)$, and thus $I$ is smoothable.
\end{proof}

\begin{prop}\label{prop:1d22-smoothable}
	All algebras in $H^d_{(1,d,2,2)}$ are smoothable.
\end{prop}
\begin{proof}

	Let $I$ be a generic ideal with Hilbert function $(1,d,2,2)$.  After a change of variable, we may assume $\left(\pi(I)\right)_2^{\perp}=\lideal y_1^2, y_2^2\rideal.$  Thus $I$ must be of the form 
	$$\lideal x_\ell^2-a_{\ell\ell}x_1^3-b_{\ell\ell}x_2^3, x_ix_j-a_{ij}x_1^3-b_{ij}x_2^3\mid i< j, 2<\ell \rideal +\m^4.$$
	Note $I$ determines  a symmetric bilinear map 
		$$\begin{array}{ccl}
		\phi\colon (\m:\m^3)/\m^2 \times (\m:\m^3)/\m^2 &\rightarrow& \m^3 \isom k^2\\
		(x_i,x_j)&\mapsto& a_{ij}x_1^3+b_{ij}x_2^3
		\end{array}$$  
	By composing $\phi$ with projections onto the two coordinates, we get a pair of symmetric bilinear forms.  For a generic $\phi$, these are linearly independent and their span is invariant under a change of basis on $\m^3$.  By~\cite[Lemma 22.42]{harris}, there exists a basis for $(\m : \m^3)/\m^2$ and  $\m^3$ such that these bilinear forms are represented by diagonal matrices.
  Thus $I$ has the following form
	$$\lideal x_\ell^2-a_{\ell}x_1^3-b_{\ell}x_2^3, x_ix_j-a_{ij}x_1^3-b_{ij}x_2^3 \mid i< j, 2<\ell \rideal +\m^4,$$
	where $a_{ij}=b_{ij}=0$ if $i$ and $j$ are both greater than $2$ and $a_{\ell},b_{\ell}$ are nonzero for all $\ell >2$.  After suitable changes of variable, we may assume $b_{ij}=a_{ij}=0$ for all $i,j$.  This gives the ideal
	$$I=\lideal x_\ell^2-a_\ell x_1^3-b_{\ell}x_2^3, x_ix_j, x_1^4, x_2^4\mid i< j,  2<\ell \rideal.$$

Now consider the following ideals:
	\begin{align*}
		J_1&:=\lideal x_\ell^2-a_\ell x_1^2-b_\ell x_2^3, x_ix_j, x_1^3,x_2^4 \mid i < j, 2<\ell \rideal,\\
		J_2&:=\lideal x_1+1,x_2,\ldots,x_d\rideal.
	\end{align*}
Note $J_1$ is a $(1,d,2,1)$ ideal and in fact lies in the component $Q_d$, and therefore is smoothable by Propositions~\ref{prop:B1d21-smoothable}.  One can check that $I=\initial_{(2,2,3,\ldots,3)}\left(J_1\cap J_2\right)$, and therefore $I$ is smoothable.
\end{proof}

\begin{prop}\label{prop:134-smoothable}
	All algebras in $H^3_{(1,3,4)}$ are smoothable.
\end{prop}

\begin{proof}
	Such algebras are given by a $2$-dimensional subspace of the space of quadratic forms, with isomorphisms given by the action of $\GL_3$. Arguing as in Proposition \ref{prop:1d2-smoothable}, we conclude that, up to isomorphism, a generic $2$-dimensional space of quadrics is spanned by $x^2+z^2$ and $y^2+z^2$. Adding the necessary cubic generators, we get that $I = \lideal y^2+z^2, x^2+z^2, z^3, yz^2, xz^2, xyz \rideal$ is a generic point of $H^3_{(1,3,4)}$.

Consider $J =\lideal y^2+z^2, x+x^2+z^2, z^3, yz^2, xz^2, xyz \rideal$. Note that $J$ is the intersection of an ideal of colength $3$ and an ideal of colength $5$:
	\begin{equation*}
	J = \lideal x+1, y^2, yz, z^2 \rideal \cap \lideal x+z^2, y^2+z^2, z^3, yz^2 \rideal
	\end{equation*}
Since both ideals in the above intersection are smoothable, $J$ itself is smoothable.  One can check that $I=\initial_{(1,1,1)}(J)$.  Therefore $I$ is smoothable.
\end{proof}

\subsection{Algebras which are smooth and smoothable}

In this section we show that the remaining irreducible sets:
$$H^3_{(1,3,3)}, H^3_{(1,3,3,1)}, U, Z$$
are in the smoothable component by finding a point in each which
is smoothable and a smooth point on the Hilbert scheme.
The following result is well known (e.g.~\cite[Lemma~18.10]{cca}
in characteristic~$0$), but we give the proof in arbitrary characteristic for
the reader's convenience:
\begin{prop} \label{prop:monomial-radical}
All monomial ideals are smoothable.
\end{prop}
\begin{proof}
Suppose we have a monomial ideal of colength $n$, written in multi-index notation $I = \lideal \vec x^{\alpha^{(1)}}, \ldots, \vec x^{\alpha^{(m)}} \rideal$. Since $k$ is algebraically closed, we can pick an arbitrarily long sequence $a_1, a_2, \dots$ consisting of distinct elements in $k$. Define
\begin{equation*}
f_i = \prod_{j = 1}^d \big((x_j - a_1)(x_j - a_2) \cdots (x_j - a_{\alpha^{(i)}_j})\big)
\end{equation*}
Note that $\initial(f_i) = \vec x^{\alpha^{(i)}}$ with respect to any global term order. Let $J$ be the ideal generated by the $f_i$ for $1 \leq i \leq m$ and then $\initial(J) \supset I$ and so $J$ has colength at most $n$. However, for any standard monomial $x^\beta$ in $I$, we have a distinct point $(a_{\beta_1},\ldots,a_{\beta_d})$ in $\mathbb A^d$, and each $f_i$ vanishes at this point. Therefore, $J$ must be the radical ideal vanishing at exactly these points and have initial ideal $I$. Thus, $I$ is smoothable.
\end{proof}

The tangent space of an ideal $I$ in the Hilbert scheme is isomorphic to $\Hom_S(I, S/I)$.  We use this fact to compute the dimension of the tangent space of a point $I$.

\begin{propo} \label{prop:133-smoothable}
All algebras in $H^3_{(1,3,3)}$ are smoothable.
\end{propo}

\begin{proof}
This irreducible set includes the smoothable monomial ideal $I=\lideal x^2, y^2, z^2, xyz\rideal$.  A direct computation shows $I$ has a $21$-dimensional tangent space, so $I$ is a smooth point in $H_n^d$.   Thus, any algebra in $H^3_{(1,3,3)}$ is smoothable.
\end{proof}

\begin{propo} \label{prop:1331-smoothable}
All algebras in $H^3_{(1,3,3,1)}$ are smoothable.
\end{propo}

\begin{proof}
The ideal $I = \lideal x^2, y^2, z^2\rideal$ in this locus is smoothable by Proposition~\ref{prop:monomial-radical}, and one can check that the Hilbert scheme is smooth at this point as well. Therefore $H^3_{(1,3,3,1)}$ is contained in the smoothable component of the Hilbert scheme.
\end{proof}

\begin{propo}\label{prop:Z-smoothable}
All algebras in $Z$ are smoothable.
\end{propo}

\begin{proof}
	Consider $I=\lideal x^2,xy,xz,yz,z^3-y^4\rideal \in  Z$ and note that $$I=\initial_{(1,0,0)} \left(\lideal x+1,y,z \rideal \cap \lideal x, yz, z^3-y^4 \rideal\right).$$  The second ideal is smoothable by Theorem~\ref{thm:plane}, so $I$ is smoothable.  One can also check $I$ is smooth in the Hilbert scheme by computing the dimension of $\Hom_S(I, S/I)$.  Therefore $Z$ is contained in the smoothable component of the Hilbert scheme.
\end{proof}

\begin{propo}\label{prop:U-smoothable}
	All algebras in $U\subset H^3_{(1,3,2,1,1)}$ are smoothable.
\end{propo}
\begin{proof}
Consider the ideal $I = \lideal x^2, xy-z^4, y^2-xz,yz \rideal \in U$.  One can check that $$I=\initial_{(7,5,3)}\left(\lideal x,y,z-1 \rideal \cap \lideal x^2,xy-z^3, y^2-xz,yz \rideal \right).$$  The second ideal in the intersection is in $Q_3$ and therefore smoothable by Proposition~\ref{prop:B1d21-smoothable}.  Therefore $I$ is smoothable by the same argument in the proof of Proposition~\ref{prop:134-smoothable}.  One can also check $I$ has a $24$-dimensional tangent space in the Hilbert scheme and is thus smooth.  Therefore $U$ is contained in the smoothable component.
\end{proof}
 
\begin{table}
\small 
\begin{tabular}[h]{|l|l||l|l||l|l|}
\hline
degree & Hilbert & $\mathcal H_{\vec h}^d$ component & reference &
$H_{\vec h}^d$ component & smoothability \\
& function $\vec h$ & dimensions & & dimensions & reference \\
\hline
$4$ & $1,3$ & $0$ & Prop~\ref{prop:stdgraded-grassmannian} & $0$ &
Prop~\ref{prop:monomial-radical} \\
$5$ & $1,3,1$ & $5$ & Prop~\ref{prop:stdgraded-grassmannian} & $5$ &Prop~\ref{prop:1d1-smoothable} \\
 & $1,4$ & $0$ & Prop~\ref{prop:stdgraded-grassmannian}& $0$ &
Prop~\ref{prop:monomial-radical} \\
$6$ & $1,3,1,1$ & $2$ & Prop~\ref{prop:multigrad-1d1s}& $7$ &
Prop~\ref{prop:1d2-smoothable} \\
 & $1,4,1$ & $9$ & Prop~\ref{prop:stdgraded-grassmannian}& $9$ &
Prop~\ref{prop:1d1-smoothable}\\
 & $1,5$ & $0$ & Prop~\ref{prop:stdgraded-grassmannian} & $0$ &
Prop~\ref{prop:monomial-radical}\\
$7$ & $1,3,1,1,1$ & $2$ & Prop~\ref{prop:multigrad-1d1s}& $9$ &
Prop~\ref{prop:1d1-smoothable} \\
 & $1,3,2,1$ & $5$, $6$ & Thm~\ref{thm:multigrad-1d21} & $9$, $10$ &
Prop~\ref{prop:B1d21-smoothable},~\ref{prop:S1d21-smoothable}\\
 & $1,3,3$ & $9$ & Prop~\ref{prop:stdgraded-grassmannian} & $9$ &
Prop~\ref{prop:133-smoothable}\\
 & $1,4,1,1$ & $3$ & Prop~\ref{prop:multigrad-1d1s} & $12$ &
Prop~\ref{prop:1d1-smoothable}\\
 & $1,4,2$ & $16$ & Prop~\ref{prop:stdgraded-grassmannian} & $16$ &
Prop~\ref{prop:1d2-smoothable}\\
 & $1,5,1$ & $14$ & Prop~\ref{prop:stdgraded-grassmannian} & $14$ &
Prop~\ref{prop:1d1-smoothable}\\
 & $1,6$ & $0$ & Prop~\ref{prop:stdgraded-grassmannian} & $0$ &
Prop~\ref{prop:monomial-radical}\\
$8$ & $1,3,1,1,1,1$ & $2$ & Prop~\ref{prop:multigrad-1d1s} & $11$ &
Prop~\ref{prop:1d1-smoothable}\\
 & $1,3,2,1,1$ & $6$ & Prop~\ref{prop:stdgraded-1de11} & $11$(?), $12$ &
Prop~\ref{prop:U-smoothable},~\ref{prop:Z-smoothable}\\
 & $1,3,2,2$ & $4$ & Prop~\ref{prop:stdgraded-1d22}& $12$ &
Prop~\ref{prop:1d22-smoothable}\\
 & $1,3,3,1$ & $9$ & Prop~\ref{prop:stdgraded-1331}& $12$ &
Prop~\ref{prop:1331-smoothable} \\
 & $1,3,4$ & $8$ & Prop~\ref{prop:stdgraded-grassmannian}& $8$ & Prop~\ref{prop:134-smoothable} \\
 & $1,4,1,1,1$ & $3$ & Prop~\ref{prop:multigrad-1d1s} & $15$ &
Prop~\ref{prop:1d1-smoothable}\\
 & $1,4,2,1$ & $7$, $11$ & Thm~\ref{thm:multigrad-1d21}& $15$, $19$ &
Prop~\ref{prop:B1d21-smoothable},~\ref{prop:S1d21-smoothable} \\
 & $1,4,3$ & $21$ & Prop~\ref{prop:stdgraded-grassmannian} & $21$ & *\\
 & $1,5,2$ & $26$ & Prop~\ref{prop:stdgraded-grassmannian} & $26$ &
Prop~\ref{prop:1d2-smoothable}\\
 & $1,5,1,1$ & $4$ & Prop~\ref{prop:multigrad-1d1s} & $18$ &
Prop~\ref{prop:1d1-smoothable}\\
 & $1,6,1$ & $20$ & Prop~\ref{prop:stdgraded-grassmannian} & $20$ &
Prop~\ref{prop:1d1-smoothable}\\
  & $1,7$ & $0$ & Prop~\ref{prop:stdgraded-grassmannian} & $0$ &
Prop~\ref{prop:monomial-radical}\\
\hline
\end{tabular}
\caption{Summary of the decomposition of Hilbert schemes by Hilbert function of
the local algebra with $h_1 \geq 3$. 
The dimensions of the components of $H_{\vec h}^d$ are computed using
Propositions~\ref{prop:fibration} and~\ref{prop:1d11-irred}. In the case
of $\vec h = (1,3,2,1,1)$, Lemmas~\ref{lem:13211z-irred}
and~\ref{lem:13211u-irred} show that $H_{\vec h}^d$ is the union of two
irreducible sets, but we don't know whether the smaller set is contained in the
closure of the larger one.} \label{tbl:thm-1}
\end{table}

\begin{thm} \label{thm:smoothable-summary}
With the exception of local algebras with Hilbert function $(1,4,3)$, every algebra with $n \leq 8$  is smoothable.
\end{thm}

\begin{proof}
The possible Hilbert functions are exactly the Hilbert functions of monomial
ideals, and for $d$ at least~$3$, one can check that there are no possibilities
other those listed in Table~\ref{tbl:thm-1}. For $d$ at most $2$, smoothability
follows from Theorem~\ref{thm:plane}
\end{proof}

In particular, this implies that there are no components other than the ones listed in Theorem~\ref{thm:irred}.

%% file: waldo.tex
\ifx\thepage\undefined\def\jobname{hilb}\input{hilb}\fi
\cvs $Date: 2008/06/19 06:54:28 $ $Author: dustin $
\section{Characterization of smoothable points of \texorpdfstring{$H^4_8$}{H48}}
\label{sec:waldo}
\cvsversion
In this section $k$ will denote a field of characteristic not~$2$ or~$3$, except for Section~\ref{subsec:upper-bound-degree} where $k=\mathbb C$.

We show that besides the smoothable component, the Hilbert scheme $H_8^4$ contains a second component parametrizing the local algebras with $\vec h = (1,4,3)$.  We prove that the intersection of the two components can be described as in Theorem~\ref{thm:waldo}, and as a result we determine exactly which algebras with Hilbert function $(1,4,3)$ are smoothable.

In Section~\ref{subsec:gl4} we introduce and investigate the Pfaffian which appears in 
Theorem~\ref{thm:waldo}, and we prove the crucial fact that it is the unique
$\GL_4$-invariant of minimal degree. In Section~\ref{subsec:naive-intersection}, we give a first approximation of the intersection locus. We then use the uniqueness results from Section~\ref{subsec:gl4} to prove Theorem~\ref{thm:waldo}. We begin by proving reducibility,
\begin{prop}\label{prop:h48red} 
For $d$ at least $4$, the Hilbert scheme $H^d_8$ is reducible.
\end{prop}
\begin{proof}
It is sufficient to find a single ideal whose tangent space dimension is less
than $8d = \dim R_8^d$.  Consider the ideal 
\begin{equation*}
J=\langle x_1^2,x_1x_2,x_2^2,x_3^2,x_3x_4,x_4^2,x_1x_4+x_2x_3\rangle + \langle x_i \mid 4 < i \leq d\rangle
\end{equation*}
The tangent space of $J$ in $H^d_8$ can be computed as $\dim_k \Hom_S(J,S/J)$.  A direct computation shows that an arbitrary element of $\Hom_S(J,S/J)$ can be represented as a matrix
\begin{equation*}
\bordermatrix{
& x_1^2 & x_1x_2 & x_2^2 & x_3^2 & x_3x_4 & x_4^2 & x_1x_4 + x_2x_3 & x_i \cr
1 & 0 & 0 & 0 & 0 & 0 & 0 & 0 & * \cr
x_1 & 2a_2 & a_1 & 0 & 0 & 0 & 0 & a_4 & * \cr
x_2 & 0 & a_2 & 2a_1 & 0 & 0 & 0 & a_3 & * \cr
x_3 & 0 & 0 & 0 & 2a_3 & a_4 & 0 & a_1 & * \cr
x_4 & 0 & 0 & 0 & 0 & a_3 & 2a_4 & a_2 & * \cr
x_1x_3 & * & * & * & * & * & * & * & * \cr
x_1x_4 & * & * & * & * & * & * & * & * \cr
x_2x_4 & * & * & * & * & * & * & * & *
}
\end{equation*}
where $i$ again ranges over $4 < i \leq d$, the $a_i$ are any elements in $k$, and each $*$ represents an independent choice of an element of $k$. Thus, $\dim_k \Hom(J,S/J) = 4 + 21 + 8(d-4) = 8d-7$.  The computation holds in all characteristics.  Since $8d-7<8d=\dim(R^d_8)$, we conclude that $J$ is not smoothable and that $H^d_8$ is reducible.
\end{proof}

\begin{remark}
The above proposition holds with the same proof even when $\ch(k)=2$ or~$3$.
\end{remark}

\subsection{A \texorpdfstring{$\GL_4$}{GL4}-invariant of a system of three quadrics}\label{subsec:gl4}

In this section we study the Pfaffian which appears in the statement of Theorem \ref{thm:waldo}.  Let $G_0$ be the standard graded Hilbert scheme $\mathcal H^4_{(1,4,3)}\cong \Gr(7,S_2)$.  Recall that any $I\in G_0$ defines a $3$-dimensional subspace $I_2^{\perp}\subset S_2^*$.  

Let $Q_1,Q_2,Q_3$ be a basis of quadrics for $I_2^{\perp}$, let $A_1,A_2,A_3$ be the symmetric $4\times 4$ matrices which represent the $Q_i$ via $\vec{y}^tA_i\vec{y}=Q_i$ where $\vec{y}$ is the vector of formal variables $(y_1,y_2,y_3,y_4)$. 
\begin{defn}
The \defi{Salmon-Turnbull Pfaffian} is the Pfaffian (i.e.\ the square root of the determinant) of the following skew-symmetric $12\times 12$ matrix:
$$
M_I=\begin{pmatrix}
0&A_1&-A_2\\
-A_1&0&A_3\\
A_2&-A_3&0
\end{pmatrix}
$$
\end{defn}
\begin{lemma}\label{lem:Newgl4inv}
The Salmon-Turnbull Pfaffian of $M_I$ coincides up to scaling with the Pfaffian of the following skew-symmetric bilinear form:
\begin{align*}
\langle, \rangle_I\colon (S_1\otimes S_2/I_2)^{\otimes 2} &\to  \bigwedge^3 S_2/I_2 \cong k \\
\langle l_1\otimes q_1, l_2\otimes q_2 \rangle_I &= (l_1l_2)\wedge q_1\wedge q_2
\end{align*}
In particular, the vanishing of the Salmon-Turnbull Pfaffian is independent of the choice of basis of $I_2^*$ and invariant under the $\GL_4$ action induced by linear change of coordinates on $S$.
\end{lemma}
\begin{proof}
Let $m_1, m_2, m_3$ be any basis of $S_2/I_2$ and let $x_1, x_2, x_3, x_4$ be a basis for $S_1$. Let $A_i$ be the matrix representation with respect to this basis of the symmetric bilinear form obtained by composing multiplication $S_1 \otimes S_1 \rightarrow S_2/I_2$ with projection onto $m_i$. Note that if $m_1, m_2, m_3$ form a basis dual to $\frac{1}{2}Q_1, \frac{1}{2}Q_2, \frac{1}{2}Q_3$ then this definition of $A_i$ agrees with the the definition of $A_i$ above. Thus, $x_j x_{j'} = \sum_i (A_i)_{j j'} m_i$ where $(A_i)_{j j'}$ is the $(j,j')$ entry of $A_i$.  Then we will use the basis $x_1 \otimes m_3, x_2 \otimes m_3, \ldots, x_4 \otimes m_1$ for $S_1 \otimes S_2/I_2$. We compute the matrix representation of $\langle,\rangle_I$ in this basis:
\begin{align*}
\langle x_j \otimes m_i, x_{j'} \otimes m_{i'} \rangle_I &= (x_j x_{j'}) \wedge m_i \wedge m_{i'} \\
&= \left(\sum_{1 \leq \ell \leq 3} (A_\ell)_{jj'} m_\ell \right) \wedge m_i \wedge m_i' \\
\intertext{If $i = i'$, this quantity will be zero. Otherwise, let $i''$ be the index which is not $i$ or $i'$ and then we get} 
&= (A_{i''})_{jj'} m_{i''} \wedge m_i \wedge m_{i'}
= \pm (A_{i''})_{jj'} m_1 \wedge m_2 \wedge m_3 
\end{align*}
where $\pm$ is the sign of the permutation which sends $1, 2, 3$ to $i'', i, i'$ respectively. Thus, with $m_1 \wedge m_2 \wedge m_3$ as the basis for $\bigwedge^3 S_2/I_2$,  $\langle,\rangle_I$ is represented as:
\begin{equation*}
\begin{pmatrix}
0 & A_1 & -A_2 \\
-A_1 & 0 & A_3 \\
A_2 & -A_3 & 0
\end{pmatrix} \qedhere
\end{equation*}
\end{proof}

Since the vanishing of the Salmon-Turnbull Pfaffian depends only on the vector subspace $I_2^\perp \subset S_2^*$, it defines a function $P$ on $G_0$ which is homogeneous of degree $2$ in the Pl\"ucker coordinates.  We next show that the Salmon-Turnbull Pfaffian is irreducible and that, over the complex numbers, it is uniquely determined by its degree and $\GL_4$-invariance.

\begin{lemma}\label{lem:mindeg}
There are no polynomials of degree $1$ in the Pl\"ucker coordinates of $G_0$ whose vanishing locus is invariant under the action of the algebraic group $\GL_4$.   Therefore, the Salmon-Turnbull Pfaffian is irreducible.
\end{lemma}
\begin{proof} We may prove this lemma by passing to the algebraic closure, and we thus assume that $k$ is algebraically closed.  Let $W=\bigwedge^3 S_2^*$ and consider the Pl\"ucker embedding of $\Gr(3,S_2^*)$ in $\mathbb{P}(W)=\Proj(R)$ where $R$ is the polynomial ring $k[p_{ij\ell}]$ where $\{i,j,\ell\}$ runs over all unordered triplets of monomials in $S_2^*$. The Pl\"ucker coordinate ring $A$ is the quotient of $R$ by a homogeneous ideal $J$.  In each degree $e$, we obtain a split exact sequence of $\GL_4$-representations:
\[0\rightarrow J_e\rightarrow \Sym_e(W)\rightarrow A_e\rightarrow 0\]
Since $J_1=0$ we have $\Sym_1(W)=A_1$, and it suffices to show that this has no $1$-dimensional subrepresentations.  Given a monomial $i\in S_2^*$ let $\alpha_i\in \mathbb N^4$ be its multi-index. For $\theta=(\theta_1, \ldots, \theta_4)$, let $L$ be the diagonal matrix with $L_{mm}=\theta_m$. The action of $L$ on the Pl\"ucker coordinate $p_{ij\ell}$ is to scale it by $\theta^{\alpha_i+\alpha_j+\alpha_\ell}$. 

Suppose that there exists an invariant polynomial $F=\sum c_{ij\ell}p_{ij\ell}$ in ${\Sym}_1(W)$.  Then $L\cdot F=\lambda F$ for some $\lambda\in k^{\times}$. But since $L\cdot F=\sum c_{ij\ell}\theta^{\alpha_i + \alpha_j + \alpha_\ell} p_{ij\ell}$ it follows that whenever $c_{ij\ell}$ and $c_{i'j'\ell'}$ are both nonzero, then $\alpha_i + \alpha_j + \alpha_\ell = \alpha_{i'} + \alpha_{j'} + \alpha_{\ell'}$.  However there are no multi-indices of total degree $6$ which are also symmetric in $\theta_1, \theta_2, \theta_3, \theta_4$.  Thus each $c_{ij\ell}=0$ and there are no nontrivial $\GL_4$-invariant polynomials of degree $1$. In particular no product of linear polynomials is $\GL_4$-invariant, and thus the Salmon-Turnbull Pfaffian is irreducible.
\end{proof}

\begin{lemma}\label{lem:mindeg2}
If $k=\mathbb{C}$ then there is exactly one polynomial of degree $2$ in the Pl\"ucker coordinates, whose vanishing locus is $\GL_4$-invariant, namely the Salmon-Turnbull Pfaffian.
\end{lemma}
\begin{proof}
We take the same notation as in the proof of Lemma \ref{lem:mindeg} and recall that we have a split exact sequence of $\GL_4(\mathbb C)$-representations:
\begin{equation*}
0 \rightarrow J_2 \rightarrow \Sym_2(W) \rightarrow A_2 \rightarrow 0
\end{equation*}
We determine the irreducible subrepresentations of $\Sym_2(W)$ by computing the following Schur function decomposition of its character $\chi$:
\[\chi=s_{(8,2,2)}+s_{(7,4,1)}+2s_{(7,3,1,1)}+s_{(7,2,1,1)}+s_{(6,6)}+3s_{(6,4,2)}+s_{(6,4,1,1)}+\]
\[2s_{(6,4,1,1)}+2s_{(6,3,2,1)}+s_{(6,2,2,2)}+2s_{(5,5,1,1)}+s_{(5,4,3)}+s_{(5,4,2,1)}+s_{(5,3,3,1)}+s_{(4,4,4)}+\]
\[s_{(4,4,3,1)}+2s_{(4,4,2,2)}+s_{(3,3,3,3)}\]
We conclude that ${\Sym}_2(W)$ contains a unique $1$-dimensional subrepresentation with character $s_{(3,3,3,3)}$. It follows from this and Lemma \ref{lem:Newgl4inv} that, over $\mathbb C$, the Salmon-Turnbull Pfaffian is the only $\GL_4$-invariant of degree $2$ in the Pl\"ucker coordinates.
\end{proof}

\begin{remark}
Salmon gives a geometric description of the Salmon-Turnbull Pfaffian~\cite[pp.242-244]{salmon}, where he shows that the Pfaffian vanishes whenever there exists a cubic form $C$ and three linear differential operators $d_1, d_2, d_3$ such that $d_iC=Q_i$. Turnbull also describes this invariant in his study of ternary quadrics~\cite{turnbull}.
\end{remark}

\subsection{The irreducible components of \texorpdfstring{$H^4_8$}{H48}}\label{subsec:h48}
Consider $H_8^4$ with its associated ideal sheaf $\mathcal I$ 
and let $\mathcal A = \mathcal O_{H_8^4}[x_1,
\ldots, x_4]/\mathcal I$. On each open affine $U = \Spec B$ such that $\mathcal
A|_U$ is free, define $f_i \in B$ to be $\frac{1}{8}\operatorname{tr}(X_i)$ where $X_i$ is the operator on the free $B$-module $\mathcal A(U)$ defined by multiplication by $x_i$.
We think of the $f_i$ as being the ``center of mass'' functions for the
subscheme of $\mathbb A_B^4$ defined by $\mathcal I|_U$.  Note that the definitions of $f_i$
commute with localization and thus they lift to define elements $f_i \in
\Gamma(H_8^4, \mathcal O_{H_8^4})$, which determine a morphism $f\colon
H_8^4 \rightarrow \mathbb A^4$.

Considered as an additive group, $\mathbb A^4$ acts on $H_8^4$
by translation. We define the ``recentering'' map $r$ to be the composition
\begin{equation*}
r\colon H_8^4 \stackrel{-f \times \operatorname{id}}{\longrightarrow} \mathbb
A^4 \times H_8^4 \longrightarrow H_8^4
\end{equation*}

By forgetting about the grading of ideals, we have a
closed immersion $\iota$ of $\Gr(7,S_2)\cong \mathcal H^4_{(1,4,3)}$ into $H^4_8$~\cite[Prop~1.5]{haiman-sturmfels}. We define $G$ to
be the preimage of this closed subscheme via the ``recentering'' map, i.e. the
fiber product:
\begin{equation*}
\begin{CD}
G @>>> H_8^4 \\
@VVV @VVrV \\
\Gr(7,S_2) @>\iota >> H_8^4
\end{CD}
\end{equation*}
We define intersections $W:=G\cap R^4_8$ and $W_0:=G_0\cap R^4_8$.  We will focus on $W_0$, and the following lemma shows that this is sufficient for describing $W$.

\begin{lemma} \label{lem:G0andG}
We have isomorphisms $G\cong \Gr(7,S_2) \times \mathbb A^4$ and $W\cong
W_0\times \mathbb A^4$.
\end{lemma}
\begin{proof}
We have a map $G \rightarrow \Gr(7,S_2)$, and 
a map $f\colon H^4_8 \rightarrow \mathbb A^4$. We claim that the
induced map $\phi\colon G \rightarrow \mathbb A^4 \times \Gr(7,S_2)$ is an
isomorphism.

Define $\psi\colon \mathbb A^4 \times \Gr(7,S_2) \rightarrow H_8^4$ to be the
closed immersion $\iota$ followed by translation.  We work with an open
affine $U \isom \Spec A \subset \Gr(7,S_2)$ such that the restriction
$\iota|_{U}$ corresponds to an ideal $I \subset A[x_1, \ldots, x_4]$ whose
cokernel is a graded $A$-module with free components of ranks $(1,4,3)$. The map
$\psi|_{\mathbb A^4 \times U}$ corresponds to the ideal $I' \subset A[t_1,
\ldots, t_4][x_1',\ldots,x_4']$ where $I'$ is the image of $I$ under the
homomorphism of $A$-algebras that sends $x_i$ to $x_i' + t_i$.  Then
$A[t_1,\ldots,t_4][x_1',\ldots,x_4']/I' \isom (A[x_1, \ldots, x_4]/I)[t_1,
\ldots, t_4]$ is a graded $A[t_1,\ldots,t_4]$-algebra with $x_i' - t_i = x_i$
homogeneous of degree $1$.  The key point is that as operators on a free $A[t_1,
\ldots, t_4]$-module, the $x_i$ have trace zero, so  the trace of the $x_i'$ is
$8t_i$. Thus, $r\circ \psi\colon\mathbb A^4 \times \Gr(7,S_2) \rightarrow H_8^4$
corresponds to an ideal $I'' \subset A[t_1, \ldots, t_4][x_1'', \ldots, x_4'']$
which is the image of $I'$ under the homomorphism that takes $x_i'$ to $x_i'' -
t_i$. This is of course the extension of $I\subset A[x_1, \ldots, x_4]$ to
$A[t_1, \ldots, t_4][x_1'', \ldots, x_4'']$ with $x_i'' = x_i$, and so $I''$ has
the required properties such that $r\circ \psi$ factors through the closed
immersion $\iota$.  Thus, $\psi$ maps to $G$.  Furthermore, we see that $r\circ
\psi$ is the projection onto the first coordinate of $\mathbb A^4 \times
\Gr(7,S_2)$ and $f \circ \psi$ is projection onto the second coordinate. Thus,
$\phi \circ \psi$ is the identity.

Second, we check that the composition $\psi \circ \phi$ is the identity on $G$.
This is clear because the composition amounts to translation by $-f$ followed by
translation by $f$.

The isomorphism for $W$ follows by the same argument.
\end{proof}

\begin{lemma}\label{lem:codim1} 
$W$ and $W_0$ are prime divisors in $G$ and $G_0$ respectively.
\end{lemma}
\begin{proof} 
The point $I=\langle x_1^2, x_1x_2, x_2^2, x_3^2, x_3x_4, x _4^2, x_1x_4\rangle$ belongs to $R^4_8$ and to $G$ and has a $33$-dimensional tangent space in any characteristic.  As a result, an open set $U$ around $I$ in the Hilbert scheme is a closed subscheme of a smooth $33$-dimensional variety $Y$.  By the subadditivity of codimension of intersections, as in \cite[Thm 17.24]{harris}, it follows that every component of~$W$ through $I$ has codimension $1$ in $G$.

%

%

To show integrality, fix a monomial ideal $M_\lambda\in G_0$, and let $U_\lambda\subset H_8^4$ be the corresponding open set, as in Section~\ref{subsec:coordinates}. For any $I\in U_{\lambda}$ the initial ideal $\initial_{(1,1,1,1)}(I)\in G_0\cap U_\lambda$ and is generated by the $(1,1,1,1)$-leading forms of the given generating set of $I$ and all cubics.  Thus we may define a projection morphism $\pi\colon U_\lambda \to G_0 \cap U_\lambda$ which corresponds to taking the $(1,1,1,1)$-initial ideal.  Since $R^4_8$ is integral, so is the image $\pi(R^4_8\cap U_\lambda)=W_0\cap U_\lambda$.  Thus $W_0$ and $W_0\times \mathbb A^4\cong W$ are integral.
\end{proof}

\subsection{A first approximation to the intersection locus}\label{subsec:naive-intersection}
Any point $I$ in $W_0$ is a singular point in the Hilbert scheme. In Lemma~\ref{lem:thebeast}, we construct an equation that cuts out the singular locus over an open set of $G_0$. The local equation defines a nonreduced divisor whose support contains $W_0$. 
The following subsets of $G_0$ will be used in this section:
\begin{align*}
G_0' &:=\{I\in G_0\mid\text{the ideal $I$ is generated in degree $2$}\}\\ 
Z_1&:=G_0\mathop\backslash G_0'\\
Z_2&:=\{I\in G_0 \mid \Hom(I,S/I)_{-2}\ne 0 \} 
\end{align*}
Note that $G_0'$ is open in $G_0$ and that every ideal in $G_0'$ is generated by seven quadrics.  The set $Z_2$ will be used in Lemma \ref{lem:mult8}.  If $I$ is any ideal in $G_0$, the tangent space $\Hom_S(I,S/I)$ is graded.  The following lemma shows that if we want to determine whether $I$ is a singular point in the Hilbert scheme, then it suffices to compute only the degree $-1$ component of the tangent space.
\begin{lemma}\label{lem:hom}
For any $I\in G'_0$ we have $\dim_k \Hom_S(I,S/I)_{-1}\geq 4$, and 
$I$ is singular in $H^4_8$ if and only if $\dim_k \Hom_S(I,S/I)_{-1}\geq 5$.
\end{lemma}

\begin{proof}
Since $S/I$ is concentrated in degrees $0$, $1$ and $2$, and $I\in G'_0$ is minimally generated only in degree $2$, we have that $\Hom_S(I,S/I)$ is concentrated in degrees $0,-1,-2$.  Furthermore, since $I\in G'_0$ we have that $\dim_k \Hom_S(I,S/I)_0=21$, because any $k$-linear map $I_2\to (S/I)_2$ will be $S$-linear.  Next, note that the morphisms $t_i\colon I_2\to (S/I)_1$ mapping $q_j$ to the class of $\frac{\partial q_j}{\partial x_i}$ are $S$-linear morphisms, and thus we have $\Hom_S(I,S/I)_{-1}$ is at least $4$-dimensional.

Since the dimension of $G'_0$ is $25$, we see that $I$ is singular if $\dim_k \Hom_S(I,S/I)_{-1}>4$.  Conversely, assume for contradiction that there exists an $I$ such that $I$ is singular and dimension  of $\Hom_S(I,S/I)_{-1}$ is exactly $4$.  Since $I$ is singular, we have that $ \Hom_S(I,S/I)_{-2}$ is nontrivial.  Let $\phi\in \Hom_S(I,S/I)_{-2}$ be a nonzero map. By changing the generators of $I$ we may assume that $\phi(q_i)=0$ for $i=1, \dots, 6$ and $\phi(q_7)=1$.  Now the vector space $\langle x_1\phi, x_2\phi, x_3\phi, x_4\phi \rangle$ is a $4$-dimensional subspace of $\Hom(I,S/I)_{-1}$.  Since we have assumed that $\dim_k \Hom(I,S/I)_{-1}=4$ it must be the case that the space $\langle x_1\phi, x_2\phi, x_3\phi, x_4\phi \rangle$ equals the space $\langle t_1, \dots, t_4\rangle$.  However, this would imply that all partial derivatives of $q_1$ are zero, which is impossible.
\end{proof}
Now we will investigate those ideals which have extra tangent vectors in degree $-1$. If $\phi\colon I_2\rightarrow (S/I)_{1}$ is a $k$-linear map then $\phi$ will be $S$-linear if and only if $\phi$ satisfies the syzygies of $I$ modulo $I$. In other words, $\phi$ should belong to the kernel of:
\begin{align*}
\Hom_k(I_2, (S/I)_1) &\to \Hom_k(\Syz(I), (S/I)) \\
\phi & \mapsto \left(\sigma \mapsto \overline{\sigma(\phi)}\right)
\end{align*}
Since $I$ contains $\mathfrak{m}^3$ and is generated by quadrics, it suffices to consider linear syzygies $\sigma$ and we have an exact sequence:
\[0\to \Hom_S(I,S/I)_{-1} \longrightarrow \Hom_k(I_2,(S/I)_1)\stackrel{\psi}{\longrightarrow} \Hom_k(\Syz(I)_1, (S/I)_2)\]
where $\Syz(I)_1$ is the vector space of linear syzygies.  
We see that the $t_i$ from the previous lemma span a $4$-dimensional subspace $T$ of the kernel of $\psi$. We obtain
\[\Hom_k(I,S/I)_{-1}/ T \stackrel{\overline\psi}{\longrightarrow} \Hom_k(\Syz(I)_1, (S/I)_2) \]
and $I\in G_0'$ will be a singular point if and only if $\ker(\overline{\psi})\neq 0$. Since $I$ is generated by quadrics, it follows that $\Syz(I)_1$ has dimension $4 \cdot 7 - 20 = 8$.  Therefore $\overline{\psi}$ is a map between $24$-dimensional spaces.  Thus $\det(\overline{\psi})$ vanishes if and only if $I\in G_0'$ is a singular point in $H^4_8$. 

The global version of this determinant will give a divisor whose support contains $W_0$. On $G'_0$ we have the $\mathcal O_{G'_0}$-algebra $\mathcal S:=\mathcal O_{G'_0}[x_1, x_2, x_3, x_4]$, which is graded in the standard way, $\mathcal S=\oplus_{i} \mathcal S_i$. We have a graded universal ideal sheaf $\mathcal I=\oplus \mathcal I_i$, and a universal sheaf of graded algebras $\mathcal S/\mathcal I=\oplus_i \mathcal S_i/\mathcal I_i$.  For all $i$ the sheaves $\mathcal{S}_i$ , $\mathcal{I}_i$ and $\mathcal{S}_i/\mathcal{I}_i$ are coherent locally free $O_{G'_0}$-modules.


Let $\mu\colon \mathcal I_2\otimes \mathcal S_1\to \mathcal I_3$ be the multiplication map.  Surjectivity of this map follows from the definition of $G_0'$.  We define ${\mathcal K}_1$ to be the kernel of this map, so that we have the following exact sequence:
\begin{equation}\label{eqn:sheaf-linear-syzygies}
0\rightarrow \mathcal K_1\rightarrow \mathcal I_2\otimes \mathcal S_1\stackrel{\mu}{\rightarrow} \mathcal I_3\rightarrow 0
\end{equation}
In other words, ${\mathcal K}_1$ is the sheaf of linear syzygies.
Let $U$ be an open subset of $G'_0$ such that $\mathcal I_2|_U$ is free. Denote generators of $\mathcal I_2(U)$ by $q_1, \dots, q_7$ and thus we have the following
\begin{equation*}
{\mathcal K}_1(U)=\left\{ \left(\sum_{i=1}^7 q_i\otimes l_i\right) \mid l_i\in \mathcal S_1(U), \sum q_il_i=0\in \mathcal I_3\right\}
\end{equation*}
To simplify notation in the following lemma we write $\sheafHom$ to denote $\sheafHom_{\mathcal O_{G_0'}}$.  
\begin{lemma} \label{lem:thebeast}
The following statements hold:
\begin{enumerate}
\item\label{lem:thebeast:2} On $G'_0$ there is a morphism of locally free sheaves of ranks $28$ and $24$ respectively:
\[ h \colon \sheafHom({\mathcal I}_2, {\mathcal S}_1)\rightarrow \sheafHom({\mathcal K}_1,{\mathcal S}_2/{\mathcal I}_2)\]
such that for any $I\in G'_0$, we have $\ker(h\otimes k(I))=\Hom(I,S/I)_{-1}$.

\item\label{lem:thebeast:3}
There is a locally free subsheaf of rank four $\mathcal T\subset \ker(h)$ inducing a morphism:
\[ \overline{h} \colon \sheafHom({\mathcal I}_2, {\mathcal S}_1)/{\mathcal T}\rightarrow \sheafHom({\mathcal K}_1,{\mathcal S}_2/{\mathcal I}_2)\]
such that $\ker(\overline{h}\otimes k(I))\ne 0$ if and only if $\dim_k \Hom(I,S/I)_{-1} \geq 5$.
\end{enumerate}
\end{lemma}

\begin{proof}
$(\ref*{lem:thebeast:2})$  We have a map of locally free $\mathcal O_{G'_0}$-modules: ${\mathcal K}_1\to \mathcal I_2\otimes \mathcal S_1$.  This induces the map
${\mathcal K}_1\otimes \check{{\mathcal S}_1}\to \mathcal I_2$.
Applying $\sheafHom(-, \mathcal S_1)$ to both sides we get:
$$\sheafHom(\mathcal I_2, \mathcal S_1)\to \sheafHom({\mathcal K}_1\otimes \check{{\mathcal S}_1}, \mathcal S_1)\cong \sheafHom({\mathcal K}_1, \mathcal S_1\otimes \mathcal S_1)$$
For the isomorphism above, we are using identities about $\sheafHom$,
tensor product of sheaves, and sheaf duality from~\cite[p. 123]{hart}.  The sequence $\mathcal S_1\otimes \mathcal S_1\to \mathcal S_2 \to \mathcal S_2/\mathcal I_2$ gives a map from $\sheafHom(\mathcal K_1, \mathcal S_1\otimes \mathcal S_1)\to \sheafHom(\mathcal K_1, \mathcal S_2/\mathcal I_2)$.  By composition we obtain the desired map $h$:
$$h\colon\sheafHom(\mathcal I_2, \mathcal S_1)\to \sheafHom({\mathcal K}_1, \mathcal S_2/\mathcal I_2)$$

Let us take a moment and consider $h$ in concrete terms, since this will be used
for proving part (\ref*{lem:thebeast:3}) of the lemma.  Let $U\subset G'_0$ be an open subset such that all relevant locally free sheaves are in fact free.  Let $q_1, \ldots, q_7$  be the generators of $\mathcal I_2(U)$ and let $\sigma_j:=\sum_{i=1}^7 q_i\otimes l_{ij}$ for $1 \leq j \leq 8$ be the generators of ${\mathcal K}_1(U)$.  Finally, let $\phi\in \Hom(\mathcal I_2, \mathcal S_1)$ be a map $(q_i \mapsto m_i)$.  Then $h(\phi)$ is the map $$\sigma_j \mapsto \sum \overline{m_il_i}$$
where $\overline{m_il_i}$ is the reduction of $m_il_i$ modulo $\mathcal I_2$.

(\ref*{lem:thebeast:3})  Over any $U$ where $\mathcal I_2$ is free, let $q_1, \dots, q_7$ the global generators.  Then we define $t_1 \colon q_i\mapsto \frac{\partial}{\partial x_1}q_i$, and we define $t_2, t_3, t_4$ similarly.  This defines a locally free subsheaf $\mathcal T(U):=\langle t_1, \dots, t_4 \rangle \subset \sheafHom(\mathcal I_2, \mathcal S_1)$ of rank $4$. By the proof of Lemma~\ref{lem:hom}, the injection $\mathcal T\to \sheafHom(\mathcal I_2, \mathcal S_1)$ remains exact under pullback to a point.  It follows that the quotient $\sheafHom(\mathcal I_2, \mathcal S_1)/\mathcal T$ is locally free of rank 24~\cite[Ex II.5.8]{hart}.

It remains to show that $\mathcal T\subset \ker(h)$ and that $\ker(\overline{h}\otimes k(I))$ is nontrivial if and only if $\dim_k \Hom(I,S/I)_{_1} \geq 5$.  This is immediate from the discussion preceding this theorem.
\end{proof}

By the previous lemma, $\overline{h}$ is a map between locally free sheaves of rank $24$, and thus $\det(\overline{h})$ defines a divisor on~$G'_0$.  To ensure that this is the restriction of a unique divisor on~$G_0$, we need to verify that $Z_1$ and $Z_2$ are not too large. For this, we construct the rational curve $\tau\colon\mathbb{P}^1\rightarrow G_0$ 
defined for $t\neq \infty$ by:
\begin{equation} \label{eqn:rational-curve}
I_t=(x_1^2,x_2^2,x_3^2,x_4^2,x_1x_2,x_2x_3+tx_3x_4,x_1x_4+tx_3x_4)
\end{equation}

\begin{lemma}\label{lem:codim2} 
$Z_1 \cup Z_2$ is a closed set of codimension at least $2$ in $G_0$.
\end{lemma}
\begin{proof}
$Z_1$ is closed because it is the support of the cokernel of the multiplication map $\mathcal I_2 \otimes \mathcal S_1 \rightarrow \mathcal S_3$. The intersection $Z_2 \cap G_0'$ is the degeneracy locus of
\begin{equation*}
\sheafHom(\mathcal I_2, \mathcal S_0) \rightarrow \sheafHom(\mathcal K_1, \mathcal S_1) \oplus \sheafHom(\mathcal K_2, \mathcal S_2/I_2)
\end{equation*}
which is the analogue of Lemma~\ref{lem:thebeast} (\ref*{lem:thebeast:2}) for computing $\Hom(I_2, S_0)$.  Thus $Z_1\cup Z_2$ is closed in $G_0$.

Because $\Pic(G_0) = \mathbb Z$ and $G_0$ is projective, checking that the $1$-cycle $\tau$ does not intersect $Z_1 \cup Z_2$ will show that $Z_1 \cup Z_2$ has codimension at least~$2$. By passing to the algebraic closure, we can assume that $k$ is algebraically closed.
The group $k^\times$ acts on $\mathbb{A}^4$ by $\alpha \cdot (x_1,x_2,x_3,x_4)=(x_1,x_2,\alpha x_3,\alpha x_4)$, and taking $\alpha = t$ maps $I_1$ to $I_t$, for any $t$ other than $0$ or $\infty$. Thus, it suffices to check that the following three ideals do not intersect $Z_1 \cup Z_2$:
\begin{align*}
I_0&=(x_1^2,x_2^2,x_3^2,x_4^2,x_1x_2,x_2x_3,x_1x_4) \\
I_1&=(x_1^2,x_2^2,x_3^2,x_4^2,x_1x_2,x_2x_3+x_3x_4,x_1x_4+x_3x_4) \\
I_{\infty}&=(x_1^2,x_2^2,x_3^2,x_4^2,x_1x_2,x_2x_3-x_1x_4,x_3x_4)
\end{align*}
It is obvious that these are generated in degree $2$. A change of variables transforms $I_\infty$ to the ideal $J$ from Proposition~\ref{prop:h48red}, which is smooth, so $\Hom(I_t, S/I_t)_{-2} = 0$ for $t = 1, \infty$. One can also check that $\Hom(I_0, S/I_0)_{-2} = 0$, which holds in all characteristics because $I_0$ is a monomial ideal. Therefore, $Z_1 \cup Z_2$ has codimension at least~$2$.
\end{proof}

\begin{lemma} \label{lem:divh}
Let $D$ be the divisor on $G_0$ defined locally by $\det(\overline{h})$.  Then $W_0$ belongs to the support of $D$.
\end{lemma}  
\begin{proof}
The Hilbert scheme is singular on $W_0$, so $W_0 \cap G_0' \subset V(\det(\overline h))$. Since $W_0$ is a divisor, Lemma~\ref{lem:codim2} tells us that $W_0$ intersects $G_0'$, so the irreducibility of $W_0$ means that it is contained in $D$.
\end{proof}

\subsection{An equation for \texorpdfstring{$W_0$}{W0} over
\texorpdfstring{$\mathbb C$}{C}} \label{subsec:upper-bound-degree}

In this section, we restrict to the case $k=\mathbb C$ in order to use the representation theory of $\GL_4(\mathbb C)$.

We will use the result of Lemma~\ref{lem:divh} to give an upper bound on the degree of $W_0$ in terms of Pl\"ucker coordinates.  This leads to a proof of Theorem \ref{thm:waldo} over $\mathbb C$.  The restriction to $\mathbb C$ will be removed in the next section.

Let $H$ be an effective divisor which generates $\Pic(G_0)=\mathbb Z$.  First we compute the degree of $D$ in Pl\"ucker coordinates, using the rational curve~$\tau$.
\begin{lemma} \label{lem:rational-curve-intersection}
The curve $\tau$ has the following intersection multiplicities:
	\begin{enumerate}
		\item$\tau\cdot H=1$.
 
		\item$\tau\cdot D=16$.
	\end{enumerate}
\end{lemma}

\begin{proof}
For the first statement, let $p_1$ and~$p_2$ be the Pl\"ucker coordinates corresponding to
the $(x_1^2,x_2^2,x_3^2,x_4^2,x_1x_2,x_2x_3,x_3x_4)$- and
$({x_1^2,x_2^2,x_3^2,x_4^2,x_1x_2,x_2x_3,x_1x_4})$-minors respectively. Then
$L=V(p_1)$ does not meet $I_{t}$ at infinity. For $t\neq\infty$, we see that
$p_2(I_t)\neq 0$, so local equations for $L$ valid at all common points of $L$ and $\tau$ are given by $L=\frac{p_1}{p_2}$. Since this equation pulls back to $t$ on $\mathbb{P}^1-\infty$ the statement follows.

For the second statement, note from the proof of Lemma~\ref{lem:codim2} that $I_\infty$ is a smooth point and $\tau$ does not intersect $Z_1$ or $Z_2$. Therefore, it suffices to check the degree on the open affine defined by $t \neq\infty$.

For every $t\neq \infty$,  $I_t$ has the following $8$ linear syzygies (where $q_1,\dots, q_7$ are the generators of~$I_t$ in the order in Equation~\ref{eqn:rational-curve}).
\begin{center}
\begin{tabular}{l|l}
$\sigma_1=x_2q_1-x_1q_5$ & $\sigma_2=x_4q_1-x_1q_7+tx_3q_7-t^2x_4q_3$\\
$\sigma_3=x_1q_2-x_2q_5$ & $\sigma_4=x_3q_2-x_2q_6+tx_4q_6-t^2x_3q_4$\\
$\sigma_5=x_2q_3-x_3q_6+tx_4q_3$ & $\sigma_6=x_1q_4-x_4q_7+tx_3q_4$\\
$\sigma_7=x_3q_5-x_1q_6+tx_3q_7-t^2x_4q_3$ & $\sigma_8=x_4q_5-x_2q_7+tx_4q_6-t^2x_3q_4$\\
\end{tabular}
\end{center}
The intersection number $\tau\cdot D$ equals the degree of $\tau^*(\det (\overline{h}))$, which we compute by writing out $\tau^*(h)$ as a matrix.  Let $\phi\in \Hom(I,S/I)_{-1}$ be written as $\phi(q_i)=c_{i,1}x_1+c_{i,2}x_2+c_{i,3}x_3+c_{i,4}x_4$ and recall that, if $\sigma_j=\sum_i q_i\otimes l_{ij}$ then  $h(\phi)(\sigma_j)=\sum\overline{\phi(q_i)l_{ij}}$ where the bar indicates that we are considering the image as an element of $S_2/I_2$. The monomials $x_1x_3$, $x_2x_4$ and $x_1x_4$ are a basis of $S_2/I_2$ for $t\neq \infty$, so we can explicitly express the $h(\phi)(\sigma_j)$ as follows

\begin{center}
\begin{tabular}{l|l}
$h(\phi)(\sigma_1)$ & $-c_{5,3}x_1x_3+c_{1,4}x_2x_4+(-tc_{1,3}+tc_{5,4})x_3x_4$\\ 
$h(\phi)(\sigma_2)$ & $(tc_{7,1}-c_{7,3})x_1x_3+(c_{1,2}-t^2c_{3,2})x_2x_4+$\\&$(-tc_{1,1}+c_{1,3}+t^3c_{3,1}-t^2c_{3,3}-t^2c_{7,2}+2tc_{7,4})x_3x_4$\\
$h(\phi)(\sigma_3)$ & $c_{2,3}x_1x_3-c_{5,4}x_2x_4+(-tc_{2,4}+tc_{5,3})x_3x_4$\\
$h(\phi)(\sigma_4)$ & $(c_{2,1}-t^2c_{4,1})x_1x_3+(tc_{6,2}-c_{6,4})x_2x_4+$\\&$(-tc_{2,2}+c_{2,4}+t^3c_{4,2}-t^2c_{4,4}-t^2c_{6,1}+2t
      c6z)x_3x_4$\\
$h(\phi)(\sigma_5)$ & $ -c_{6,1}x_1x_3+(tc_{3,2}+c_{3,4})x_2x_4+(-t^2c_{3,1}+tc_{6,2}-c_{6,4})x_3x_4$\\
$h(\phi)(\sigma_6)$ & $(tc_{4,1}+c_{4,3})x_1x_3-c_{7,2}x_2x_4+(-t^2c_{4,2}+tc_{7,1}-c_{7,3})x_3x_4$\\
$h(\phi)(\sigma_7)$ & $(c_{5,1}-c_{6,3}+tc_{7,1})x_1x_3-t^2c_{3,2}x_2x_4+$\\
&$(t^3c_{3,1}-t^2c_{3,3}-tc_{5,2}+c_{5,4}+tc_{6,4}-t^2c_{7,2}+
      tc_{7,4})x_3x_4$\\
$h(\phi)(\sigma_8)$ & $-t^2c_{4,1}x_1x_3+(c_{5,2}+tc_{6,2}-c_{7,4})x_2x_4+$\\&$(t^3c_{4,2}-t^2c_{4,4}-tc_{5,1}+c_{5,3}-t^2c_{6,1}+tc_{6,3}
      +tc_{7,3})x_3x_4$\\
\end{tabular}
\end{center}
Each row of the above lines yields three linear equations so $\tau^{*}(h)$ is represented by a $24\times 28$ matrix $M$ as expected. Computation in Macaulay2~\cite{macaulay2} shows that the ideal of $24\times 24$ minors of $M$ is $(t^{16})$ and the statement follows. 
\end{proof}

\begin{cor}\label{cor:16H}
The divisor $D$ is linearly equivalent to $16H$.
\end{cor}

The following lemma allows us to determine the degree of $W_0$.
\begin{lemma}\label{lem:mult8} 
The divisor $D$ vanishes with multiplicity at least $8$ on $W_0$.
\end{lemma}
\begin{proof}
By Lemma \ref{lem:divh}, we know that $\lvert W_0\rvert\subseteq \lvert D\rvert $.  By Lemma~\ref{lem:codim2}, a general point of $W_0$ does not belong to $Z_1\cup Z_2$.  Let $I$ be any such point.  Since $I$ is a singular point in $R^4_8$, $I$ has tangent space dimension at least $\dim(R_8^4) + 1 = 33$, and so the null space of $\overline h \otimes k(I)$ must have dimension at least $8$.

Choose 8 vectors from the null space as basis vectors, and any other 16 to complete a basis of the source of $\overline{h}\otimes k(I)$.  This basis in the quotient ring lifts to a basis in the local ring $\mathcal O_{G'_0, I}$.  When we represent the localization of the map $(\overline{h})_I$ as a matrix with respect to this basis we see that the first 8 columns belong to the maximal ideal $\m_I$ of $\mathcal O_{G'_0,I}$.   Thus $\det(\overline{h})$ belongs to $\m_I^8$, and in turn $D$ has multiplicity at least 8 at $I$.
\end{proof}

\begin{lemma}\label{lem:sheafD}
The ideal sheaf of $D$ is $(P^8)$ where $P$ is the Salmon-Turnbull Pfaffian.
\end{lemma}
\begin{proof} 
Since $D$ is a divisor on $G_0$ its defining ideal in the homogeneous coordinate ring of the Pl\"ucker embedding of $G_0$ is generated by a single element $f$ of degree $16$ in the Pl\"ucker coordinates. If $g$ is the square-free part of $f$ then Lemma~\ref{lem:mult8} shows that $g$ has degree at most $2$. Since $D$ is invariant under linear changes of variables, it follows from  Lemmas~\ref{lem:mindeg} and~\ref{lem:mindeg2} that $g=P$ and $f=P^8$.
\end{proof}

By combining Lemmas~\ref{lem:codim1}, \ref{lem:divh}, and~\ref{lem:sheafD} we have have now proven our descriptions of $W_0$ and $W$:

\begin{thm}\label{thm:final}
The subscheme $W_0$ is defined by $P$.
\end{thm}

\subsection{Extension to fields other than \texorpdfstring{$\mathbb C$}{C}} \label{subsec:extension-other-fields}
Now we return to the case that $k$ is a field, not necessarily algebraically closed, of characteristic not~$2$ or~$3$.

Recall that if $M_\lambda$ is any monomial ideal in $G_0$ then there are local
coordinates $c^m_{m'}$ on $U_\lambda \cap H^4_8$.  Moreover there is a
surjection $\pi\colon R^4_8 \cap U_\lambda \to W_0 \cap U_\lambda$, and there is
a rational map $\phi\colon (\mathbb A^4)^8\mathord\sslash S_8 \dasharrow R^4_8\cap U_\lambda$ given by $c^m_{m'}=\frac{\Delta_{\lambda -m'+m}}{\Delta_{\lambda}}$ whose image is dense in $R^4_8\cap U_\lambda$.  
\begin{lemma}\label{lem:vanish}
With $U_\lambda$ as above, the function $P\circ \pi$ vanishes identically on $R^4_8\cap U_\lambda$ over an arbitrary field $k$.
\end{lemma}
\begin{proof}
The composition $P \circ \pi \circ \phi$ is a rational function with integer coefficients. Theorem~\ref{thm:final} proves that $P\circ\pi\circ\phi = 0$ in $\mathbb C[q^{(j)}_i][\Delta_{\lambda}^{-1}]$. Therefore, $P\circ\pi\circ\phi = 0$ in $\mathbb Z[q^{(j)}_i][\Delta_{\lambda}^{-1}]$.
\end{proof}

\begin{thm}\label{thm:final-all-char}
The following irreducible subsets of $G_0$ coincide:
\begin{enumerate}
\item{$W_0$} 
\item{$V(P)$, the vanishing of the pullback to $G$ of the Salmon-Turnbull Pfaffian.} 
\item{The homogeneous ideals with Hilbert function $(1,4,3)$ which are flat limits of ideals of distinct points.}
\end{enumerate}
As a consequence, if we let $\pi|_{G}\colon G \to G_0$ be the restriction of the projection from Lemma~\ref{lem:codim1} then $W=V(P\circ \pi|_{G})$.
\end{thm}

\begin{proof}
For other fields, note that for the ideal $J$ of Proposition \ref{prop:h48red} with $d=4$, we have that $P\circ \pi(J)=P(J)=1$ and thus $P\circ \pi$ does not vanish uniformly on $G$ in any characteristic.  By the previous lemma, $P\circ \pi$ vanishes uniformly on $R^4_8$ for any $k$.  Thus $W\subseteq V(P\circ \pi|_{G})$.  As both $W$ and $V(P\circ \pi|_{G})$ are integral closed subschemes of codimension $1$ in $G$, they are equal.
\end{proof}

%% file: pfofmain.tex
\ifx\thepage\undefined\def\jobname{hilb}\input{hilb}\fi
\cvs $Date: 2008/06/17 12:20:04 $ $Author: dustin $
\section{Proofs of main results}\label{sec:pfofmain}\cvsversion

In this section, $k$ denotes a field of characteristic not~$2$ or~$3$

\begin{proof}[Proof of Theorem~\ref{thm:h48red}]
The irreducibility of $H^d_n$ when $d$ is at most $3$ or $n$ is at most $7$ follows for an algebraically closed field from Theorem~\ref{thm:smoothable-summary}. For a non-algebraically closed field, the Hilbert scheme is irreducible because it is irreducible after passing to the algebraic closure. Proposition~\ref{prop:stdgraded-grassmannian-bundle} and the same argument as in Lemma~\ref{lem:G0andG} show that when $d$ is at least $4$, $G_8^d$ is irreducible and $(8d-7)$-dimensional, and Proposition~\ref{prop:h48red} shows that it is a separate component. Theorem~\ref{thm:smoothable-summary} shows that there are no other components, again, by passing to the algebraic closure if necessary. 
\end{proof}

\begin{proof}[Proof of Theorem~\ref{thm:multigrad-red}]
This follows from Theorem~\ref{thm:summary-stdgraded}.
\end{proof}

\begin{proof}[Proof of Theorem~\ref{thm:waldo}]
The statement that $R^4_8\cap G^4_8$ is a prime divisor on $G^4_8$ is proved in Lemma~\ref{lem:codim1}.  The equivalence of the set-theoretic description and the local equation description follows from Lemma~\ref{lem:Newgl4inv}.  Theorem~\ref{thm:final-all-char} proves that the Salmon-Turnbull Pfaffian is the correct local equation.
\end{proof}

\begin{proof}[Proof of Theorem~\ref{thm:eqns-radical-component}]
Let $M_\lambda$ be some monomial ideal and consider the monomial chart $U_\lambda$.  If $M_\lambda$ does not have Hilbert function $(1,4,3)$ then $U_{\lambda}\cap G^4_8=\varnothing$ so that the zero ideals will cut out $R^4_8$.  If $M_\lambda$ has Hilbert function $(1,4,3)$, then Lemma~\ref{lem:vanish}
and Theorem~\ref{thm:final-all-char} show that the zero set of the pullback of the Pfaffian is precisely $R^4_8\cap U_\lambda$.
\end{proof}

%% file: OpenQuest.tex
\ifx\thepage\undefined\def\jobname{hilb}\input{hilb}\fi
\cvs $Date: 2008/06/16 15:59:38 $ $Author: dustin $
\section{Open questions}\label{sec:OpenQuest}\cvsversion

The motivating goal behind this work is understanding the smoothable component of the Hilbert scheme as explicitly as possible, and not just as the closure of a certain set. This can be phrased more abstractly by asking what functor the smoothable component represents or more concretely by describing those algebras which occur in the smoothable component. In this paper we have done the latter for $n$ at most $8$. The following are natural further questions to ask:

\begin{itemize}
	\item For $d$ greater than $4$, which algebras with Hilbert function $(1,d,3)$ are smoothable? Generically, such algebras are not smoothable.  Computer experiments lead us to conjecture that, for smoothable algebras, the analogue of the skew symmetric matrix in Theorem~\ref{thm:waldo} has rank at most $2d+2$. However, a dimension count shows that this rank condition alone is not sufficient for such an algebra to be smoothable. What are the other conditions?
	
	\item What is the smallest $n$ such that $H^3_n$ is reducible? We have shown $H^3_8$ is irreducible and Iarrobino has shown that $H^3_{78}$ is reducible~\cite[Example 3]{iarrobino-compressed}.

	\item Is $H^d_n$ ever non-reduced? What is the smallest example? Does it ever have generically non-reduced components?
\end{itemize}